\documentclass{amsart}
\usepackage{amscd}
\usepackage{bm}
\usepackage{longtable}
\usepackage{color}
\newtheorem{theorem}{Theorem}[section]
\newtheorem{lemma}[theorem]{Lemma}

\theoremstyle{definition}
\newtheorem{definition}[theorem]{Definition}
\newtheorem{example}[theorem]{Example}
\newtheorem{claim}[theorem]{Claim}
\newtheorem{proposition}[theorem]{Proposition}
\newtheorem{corollary}[theorem]{Corollary}
\newtheorem{conjecture}[theorem]{Conjecture}
\newtheorem{convention}[theorem]{Convention}
\newtheorem{notation}[theorem]{Notation}
\newtheorem{remark}[theorem]{Remark}
\newtheorem{algorithm}[theorem]{Algorithm}

\numberwithin{equation}{section}

\newcommand{\ord}{\mbox{\rm ord}}

\newcommand{\id} {{\bf 1}}

\newcommand{\Ind}{{\mbox{\rm Ind}}}

\newcommand{\rank}{\mbox{\rm {rank}}}

\newcommand{\rad}{\mbox{\rm rad}}

\def\L{\hbox{\bf L}}
\def\I{\hbox{\bf I}}

\def\CX{\mathbb{C}}
\def\CQ{{\mathbb Q}}
\def\bG{{\mathbb G}}

\def\bX{{\mathbb X}}
\def\bY{{\mathbb Y}}

\def\diag{\hbox{\rm diag}}
\def\coeff{\hbox{\rm coeff}}
\def\GL{{\rm GL}}

\def\Sym{{\rm Sym}}

\def\stab{{\rm stab}}
\def\Zero{{\rm Zero}}

\def\Mat{{\rm Mat}}

\def\lc{{\rm lc}}

\def\ord{{\rm ord}}
\def\rI{{\rm I}}
\def\rW{{\rm W}}
\def\Spec{{\rm Spec}}
\def\max{{\rm max}}

\def\B{{\mathcal B}}
\def\G{{\mathcal G}}
\def\I{{\mathcal I}}

\def\P{{\mathcal P}}
\def\R{{\mathcal R}}

\def\L{{\mathcal L}}

\def\calU{{\mathcal U}}

\def\H{{\mathcal H}}

\def\M{{\mathcal M}}
\def\Z{\hbox{\bf Z}}
\def\bV{\mathbb{V}}

\def\bZ{\mathbb{Z}}

\def\frakX{{\mathfrak{X}}}

\def\bQ{{\bar{\mathbb{Q}}}}

\def\F{{\mathcal F}}

\def\calF{{\mathcal F}}
\def\calZ{{\mathcal Z}}

\def\calL{{\mathcal L}}

\def\calS{{\mathcal S}}
\def\bfalpha{{\bm\alpha}}
\def\bfxi{{\bm\xi}}
\def\bfmu{{\bm\mu}}
\def\bfa{{\mathbf a}}
\def\bfb{{\mathbf b}}
\def\bfc{{\mathbf c}}
\def\bfd{{\mathbf d}}
\def\bfe{{\mathbf e}}
\def\bff{{\mathbf f}}
\def\bfi{{\mathbf i}}
\def\bfj{{\mathbf j}}

\def\bfs{{\mathbf s}}

\def\bfv{{\mathbf v}}

\def\bfn{{\mathbf n}}

\def\bfh{{\mathbf h}}

\def\bfT{{\mathbf T}}

\def\bfY{{\mathbf Y}}
\newcommand{\bfeta}{{\bm\eta}}
\newcommand{\bfchi}{{\bm\chi}}
\def\frakm{{\mathfrak{m}}}
\def\frakh{{\mathfrak{h}}}
\def\frakp{{\mathfrak{p}}}

\begin{document}

\title{Difference Galois Groups under Specialization}

\author{Ruyong Feng}
\address{KLMM,Academy of Mathematics and Systems Science, and School of Mathematics, University of Chinese Academy of Sciences,
Chinese Academy of Sciences, No.55 Zhongguancun East Road, Beijing 100190, China}
\email{ryfeng@amss.ac.cn}
\thanks{The author was supported in part by NSFC Grants No.11771433 and No.11688101. The author would like to thank Michael F. Singer for many valuable conversations. Especially, he suggested considering the arguments in the proof of Theorem 4.4 of \cite{singer}. }


\date{}


\begin{abstract}
We present a difference analogue of a result given by Hrushovski on differential Galois groups under specialization. Let $k$ be an algebraically closed field of characteristic zero and $\bX$ an irreducible affine algebraic variety over $k$. Consider the linear difference equation
$$
   \sigma(Y)=AY
$$
where $A\in \GL_n(k(\bX)(x))$ and $\sigma$ is the shift operator $\sigma(x)=x+1$. Assume that the Galois group $G$ of the above equation over $\overline{k(\bX)}(x)$ is defined over $k(\bX)$ i.e. the vanishing ideal of $G$ is generated by a finite set $S\subset k(\bX)[X,1/\det(X)]$. For a $\bfc\in \bX$, denote by $v_{\bfc}$ the map from $k[\bX]$ to $k$ given by $v_{\bfc}(f)=f(\bfc)$ for any $f\in k[\bX]$. We prove that the set of $\bfc\in \bX$ satisfying that $v_\bfc(A)$ and $v_\bfc(S)$ are well-defined and the affine variety in $\GL_n(k)$ defined by $v_{\bfc}(S)$ is the Galois group of $\sigma(Y)=v_{\bfc}(A)Y$ over $k(x)$ is Zariski dense in $\bX$.

We apply our result to van der Put-Singer's conjecture which asserts that an algebraic subgroup $G$ of $\GL_n(k)$ is the Galois group of a linear difference equation over $k(x)$ if and only if the quotient $G/G^\circ$ by the identity component is cyclic. We show that if van der Put-Singer's conjecture is true for $k=\CX$ then it will be true for any algebraically closed field $k$ of characteristic zero.
\end{abstract}

\maketitle

\section{Introduction}
\label{SEC:introduction}
Let $K$ be a function field of one variable over $\CQ$ and $\L$ a linear differential operator with coefficients in the differential field $(K(t), {\rm d}/{\rm d}t)$. For a place $\frakp$ in $K$, $\Sigma_\frakp$ denotes its residue field, and $\L_\frakp$ denotes the differential operator over $\Sigma_\frakp(x)$ obtained by applying $\frakp$ to the coefficients of $\L$.
In \cite{hrushovski}, Hrushovski proved that for many places $\frakp$ in $K$, the Galois group of $\L(y)=0$ over $\bar{K}(t)$ specializes precisely to the Galois group of $\L_\frakp(y)=0$ over $\bar{\Sigma}_\frakp(t)$. As a corollary, he proved a function field analogue of Grothendieck-Katz's conjecture on $p$-curvatures. The reader is referred to \cite{katz} for this conjecture and to (\cite{divizio},\cite{pillay}) for its generalizations. In particular, Di Vizio in \cite{divizio} presented a positive answer of a $q$-analogue of Groethendieck-Katz's conjecture, i.e. an analogue statement for $q$-difference equations. The difference analogue of the Grothendieck-Katz's conjecture is not true  (see a counterexample on page 58 of \cite{vanderPut-Singer}). But one can still ask whether Hrushovski's result holds true for linear difference equations. The goal of this paper is to provide an affirmative answer of this question. Let us start with an example.
\begin{example}
\label{example1}
Let $\bX=\mathbb{A}^1(\CX)$ and denote $\CX(\bX)=\CX(t)$. Consider
$$
   \sigma(Y)=\diag(t, x, x+t)Y
$$
where $\sigma$ is the shift operator $\sigma(x)=x+1$. Denote $A(t)=\diag(t, x, x+t)$. Due to van der Put-Singer's method (see Section 2.2 of \cite{vanderPut-Singer}),  $\bG_m^3(\overline{\CX(t)})$ is the Galois group of the above equation over $\overline{\CX(t)}(x)$, where $\bG_m$ stands for the multiplicative group. Now let $c\in \mathbb{A}^1(\CX)\setminus\{0\}$. By van der Put-Singer's method again, one sees that the Galois group of $\sigma(Y)=v_c(A)Y$ over $\CX(x)$ equals $\bG_m^3(\CX)$ if and only if $c$ is neither a root of unity nor an integer. On the other hand, the vanishing ideal of $\bG_m^3(\CX)$ is generated by $S=\{X_{1,2}, X_{1,3}, X_{2,1}, X_{2,3}, X_{3,1}, X_{3,2}\}$. For any $c\in \CX$, the variety in $\GL_3(\CX)$ defined by $v_c(S)$ is $\bG_m^3(\CX)$.
\end{example}
This example implies that on the one hand there are infinitely many ``good" $c\in \mathbb{A}^1(\CX)$ such that the Galois group of $\sigma(Y)=v_c(A)Y$ over $\CX(x)$ is equal to $\bG_m^3(\CX)$, on the other hand these good $c$ do not form an open subset of $\mathbb{A}^1(\CX)$ in the sense of Zariski topology. Thus other algebraic structures rather than Zariski open sets are necessary to describe these good $c$. For this purpose, we introduce basic open subsets of the corresponding variety (see Definition~\ref{DEF:basicopensets}).

Throughout this paper, $k$ denotes an algebraically closed field of characteristic zero. Let $\bX$ be an irreducible affine algebraic variety over $k$. $k[\bX]$ (resp. $k(\bX)$) denotes the ring (resp. field) of regular (resp. rational) functions on $\bX$, and $\overline{k(\bX)}(x)$ stands for the field of rational functions in $x$ with coefficients in $\overline{k(\bX)}$, the algebraic closure of $k(\bX)$. Over $\overline{k(\bX)}(x)$, we can define a shift operator $\sigma$ as the following: $\sigma(x)=x+1$ and $\sigma(c)=c$ for all $c\in \overline{k(\bX)}$. Consider the linear difference equation
\begin{equation}
\label{EQ:differenceeqn}
 \sigma(Y)=AY
\end{equation}
where $Y$ is an $n$-vector of indeterminates and $A\in \GL_n(k(\bX)(x))$. Let $X=(X_{i,j})$ be an $n\times n$ matrix of indeterminates and $\overline{k(\bX)}(x)[X,1/\det(X)]$ denotes the ring over $\overline{k(\bX)}(x)$ generated by entries of $X$ and $1/\det(X)$. The main result of this paper is as follows.
\begin{theorem}
\label{TH:main} Suppose that $G$ is the Galois group of $\sigma(Y)=AY$ over $\overline{k(\bX)}(x)$ and the vanishing ideal of $G$ is generated by a finite set $S\subset k[\bX][X]$. Then there is a basic open subset $U$ of $\bX$ such that for any $\bfc\in U$, the variety in $\GL_n(k)$ defined by $v_{\bfc}(S)$ is the Galois group of $\sigma(Y)=v_{\bfc}(A)Y$ over $k(x)$.
\end{theorem}

We prove in Theorem~\ref{TH:basicopen} that every basic open subset of $\bX$ is Zariski dense in $\bX$.
Theorem~\ref{TH:main} together with Theorem~\ref{TH:basicopen} then gives a positive answer to the question posed at the beginning of this paper. Similar to the Hrushovski's treatment in \cite{hrushovski}, the proof of the above theorem relies on the computation of difference Galois groups and other algorithmic aspects of linear difference equations,  which are developed in \cite{feng2,petkovsek,vanderPut-Singer} etc. Our way to compute difference Galois groups is via the Picard-Vessiot theory.  Remark that there is another way so-called Tannakian category method to construct Galois groups. Based on this category approach, a similar result was obtained in \cite{braverman-etingof-gaitsgory} for differential Galois groups of quantum completely integrable systems.

Theorem~\ref{TH:main} can be applied to van der Put-Singer's conjecture concerning the inverse problem in difference Galois theory. Let $G$ be an algebraic subgroup of $\GL_n(k)$. Theorem~\ref{TH:main} allows one to conclude that if $G(\overline{k(\bX)})$ is the Galois group of a linear difference equation over $\overline{k(\bX)}(x)$ then $G$ is the Galois group of a linear difference equation with coefficients in $k(x)$. This enables us to reduce van der Put-Singer's conjecture to the case where the field of constants is the field of complex numbers. Note that in \cite{maier} the specialization technique is also applied to realize a semisimple, simply-connected linear algebraic group defined over $\mathbb{F}_q$ as a Galois group of a Frobenuis difference equation.

Recently, parameterized Galois theories were developed in \cite{cassidy-singer,divizio-hardouin-wibmer,hardouin-singer, ovchinnikov-wibmer} etc for linear difference (or differential) equations with parameters admitting actions of the derivations or endomorphisms. These parameterized Galois theories provide a powerful tool to measure the differential (or difference) dependencies among solutions of the corresponding equations and have found many applications in combinatorics and the theory of special functions. However, the present paper focuses on linear difference equations with parameters where the derivations or endomorphisms act trivially. These equations can be regarded as a family of linear difference equations parameterized by an irreducible affine variety, and the Galois groups then measure the algebraic relations among solutions at generic points. The main result of this paper tells us for what specializations of the parameters these algebraic relations among solutions are preserved precisely. From Example~\ref{example4}, one may see that there are specializations which destroy the algebraic relations completely.

The  rest of this paper is organized as follows. In Section 2, we introduce the notion of basic open subsets of an irreducible affine variety $\bX$ over $k$ and present the properties of these subsets. Sections 3 and 4 present some preliminary results for the proof of Theorem~\ref{TH:main}. In Section 3, we deal with algebraic groups defined over $k(\bX)$. Precisely, we prove that for almost all $\bfc\in \bX$, $v_\bfc$ preserves the structure of algebraic groups and is bijective from the characters of a connected algebraic group $G$ to those of $G_\bfc$, the specialized group of $G$. In Section 4, we consider $\sigma$-ideals.  We show that given a $\nu$-maximal $\sigma$-ideal of $\overline{k(\bX)}(x)[X,1/\det(X)]$ (see Definition~\ref{DEF:nu-maximalspace}) generated by a finite set $S\subset k[\bX][X]$, there is a basic open subset $U$ of $\bX$ such that $v_\bfc(S)$ generates a $\nu$-maximal $\sigma$-ideal of $k(x)[X,1/\det(X)]$ for all $\bfc\in U$. We prove Theorem~\ref{TH:main} in Section 5 and apply this theorem to the inverse problem in difference Galois theory in Section 6.

{\bf Notations}: When $P$ is an element in $k[\bX][X,1/\det(X)]$ or a matrix with entries in $k[\bX]$, we also use $P(\bfc)$ to denote $v_\bfc(P)$. All varieties in this paper will be affine.
\begin{longtable}{ll}
$k, L$ & algebraically closed fields of characteristic zero\\
$\bG_a$ (resp. $\bG_m$) & additive (resp. multiplicative) group\\
$\bX, \bY$    &affine algebraic varieties over $k$\\
$k[\bX]$ & the ring of regular functions on $\bX$\\
$k(\bX)$ & the field of rational functions on $\bX$\\
$p_{\bY/\bX}$ & the projection from $\bY$ to $\bX$ induced by $k[\bX]\subset k[\bY]$\\
$v_\bfc$ & the map from $k[\bX]$ to $k$ given by $v_{\bfc}(f)=f(\bfc)$\\
$\Gamma, \tilde{\Gamma}$ &  finitely generated subgroups of $\bG_a(\overline{k(\bX)})$  or $\bG_m(\overline{k(\bX)})$\\
$\tilde{U},U,U_1,U_2,\cdots$ & basic open subsets of $\bX$ or $\bY$\\
$\bX_f$ & $\{\bfc\in \bX |f(\bfc)\neq 0\}$, where $f\neq 0$\\
$G$         & an algebraic subgroup of $\GL_n(\overline{k(\bX)})$ (or $\GL_n(k)$)\\
$G^\circ$   & the identity component of $G$\\
$\bfchi(G)$      & the group of characters of $G$\\
$G(\overline{k(\bX)}(x))$  & the set of $\overline{k(\bX)}(x)$-points of $G$ \\
$\Z(f)$     & the set of integer zeroes of $f$\\
$X$ (resp. $Z$)        & $n\times n$ matrix with indeterminate entries $X_{i,j}$ (resp. $Z_{i,j}$)\\
$L[X]_{\leq d}$ & the set of polynomials in $L[X]$ with total degree $\leq d$
\end{longtable}

\section{Basic open subsets of $\bX$}
\label{SEC:basicopensets}
 In this section, we shall introduce an algebraic structure of $\bX$ which is Zariski dense in $\bX$ and consists of good specializations. Throughout this section, we fix an algebraic closed field $L$ containing $k(\bX)$ and all $k$-algebras will be in $L$. Assume that $\Gamma$ is a finitely generated subgroup of $\bG_a(L)$ or $\bG_m(L)$. Denote by $\bY$ the variety over $k$ associated to $k[\bX][\Gamma]$, the $k[\bX]$-algebra in $L$ generated by $\Gamma$ and denote by $p_{\bY/\bX}$ the morphism from $\bY$ to $\bX$ induced by the inclusion $k[\bX]\subset k[\bX][\Gamma]$. Note that $k[\bY]=k[\bX][\Gamma]$ and $\bY$ can be identified with the set of all $k$-homomorphisms from $k[\bX][\Gamma]$ to $k$. Under this identification, for $\bfc\in \bY$, we use $v_\bfc$ to denote the $k$-homomorphism corresponding to $\bfc$. One sees that for $f\in k[\bX][\Gamma]$, $v_\bfc(f)$ is equal to the value at $\bfc$ of $f$ viewed as a regular function on $\bY$, i.e. $v_\bfc(f)=f(\bfc)$. We are interested in those $\bfc\in \bY$ whose induced maps $v_\bfc$ are injective on $\Gamma$. Set
\begin{equation}
\label{eqn:basicopensets}
    \B(\bX,\Gamma)=p_{\bY/\bX}\left(\left\{\bfc\in \bY |
                                     \mbox{$v_\bfc$ is injective on $\Gamma$}\right\}\right).
\end{equation}
\begin{definition}
\label{DEF:basicopensets}
A basic open subset of $\bX$ is defined to be the intersection of finitely many subsets of $\bX$ of the form $\B(\bX,\Gamma)$. When $\Gamma$ is the subgroup of $\bG_a(L)$ generated by a single $g\in L$, we will abbreviate $\B(\bX,\Gamma)$ to $\bX_g$.
\end{definition}
Remark that when $\Gamma \subset k[\bX]$, one can take $\bY=\bX$ and then $p_{\bY/\bX}$ is the identity map and $\B(\bX,\Gamma)=\left\{\bfc\in \bX |\mbox{$v_\bfc$ is injective on $\Gamma$}\right\}$. The reason that $\Gamma$ is not restricted to $k[\bX]$ is as follows: On the one hand, the extension of $k[\bX]$ is necessary in some cases such as the defining field of characters of $G^\circ$ (see Example~\ref{example3}); On the other hand, if we restrict $\Gamma$ to $k[\bX]$ in Definition~\ref{DEF:basicopensets} then we do not know whether basic open sets are preserved by the projection map, although they do if they are only defined by additive groups (see Lemma 5A.1 of \cite{hrushovski}).
Two lemmas below imply that basic open sets without the above restriction are preserved by the projection map in some sense. The first one is due to Proposition 9 on page 34 of \cite{kolchin}.
\begin{lemma}
\label{LM:kolchinlemma}
Assume that $\bY$ is a variety over $k$ associated to a finitely generated $k[\bX]$-algebra in $L$.
For any $\tilde{f}\in k[\bY]\setminus\{0\}$, there is a nonzero $f\in k[\bX]$ such that
$$
   \bX_f \subset p_{\bY/\bX}(\bY_{\tilde{f}}).
$$
\end{lemma}
\begin{lemma}
\label{LM:property}
Suppose that $\bY$ is as in Lemma~\ref{LM:kolchinlemma} and $U$ is a basic open subset of $\bY$. Then $p_{\bY/\bX}(U)$
contains a basic open subset of $\bX$.
\end{lemma}
\begin{proof}
It suffices to show the assertion with $U=\B(\bY,\Gamma)$, where $\Gamma$ is a finitely generated subgroup of $\bG_a(L)$ or $\bG_m(L)$. Assume that $k[\bY]$ is generated by a finite subset $T$ of $L\setminus \{0\}$ as a $k[\bX]$-algebra. Let $\tilde{\Gamma}$ be generated by $\Gamma\cup T$ as a group of the same type as $\Gamma$. Then $k[\bY][\Gamma]\subset k[\bX][\tilde{\Gamma}]=k[\bY][\tilde{\Gamma}]$. Let $\tilde{\bY}$ and $\bY'$ be the varieties over $k$ associated to $k[\bX][\tilde{\Gamma}]$ and $k[\bY][\Gamma]$ respectively. Since $\tilde{\Gamma}\subset k[\tilde{\bY}]$, $\B(\tilde{\bY},\tilde{\Gamma})=\{\bfc\in \tilde{\bY}|\mbox{$v_\bfc$ is injective on $\tilde{\Gamma}$}\}$. Then by definition, one has that
\begin{equation}
\label{EQ:equality1}
\B(\bX,\tilde{\Gamma})=p_{\tilde{\bY}/\bX}(\{\bfc\in \tilde{\bY}|\mbox{$v_\bfc$ is injective on $\tilde{\Gamma}$}\})=p_{\tilde{\bY}/\bX}(\B(\tilde{\bY},\tilde{\Gamma})).
\end{equation}
Similarly, $U=\B(\bY,\Gamma)=p_{\bY'/\bY}(\B(\bY',\Gamma))$.
Furthermore, as the morphism $p_{\tilde{\bY}/\bY'}$ is induced by the inclusion $k[\bY][\Gamma]\subset k[\tilde{\bY}]$, for any $\bfc\in \tilde{\bY}$ and any $f\in \Gamma$, $v_{p_{\tilde{\bY}/\bY'}(\bfc)}(f)=v_\bfc(f)$. This implies that if $v_\bfc$ is injective on $\Gamma$ then so is $v_{p_{\tilde{\bY}/\bY'}(\bfc)}$. Hence $p_{\tilde{\bY}/\bY'}(\B(\tilde{\bY},\Gamma))\subset \B(\bY',\Gamma)$ and then
\begin{equation}
\label{EQ:subset}
 p_{\tilde{\bY}/\bX}(\B(\tilde{\bY},\Gamma))=p_{\bY/\bX}(p_{\bY'/\bY}(p_{\tilde{\bY}/\bY'}(\B(\tilde{\bY},\Gamma))))\subset p_{\bY/\bX}(p_{\bY'/\bY}(\B(\bY',\Gamma)))=p_{\bY/\bX}(U).
 \end{equation}
Finally as $\B(\tilde{\bY},\tilde{\Gamma})\subset \B(\tilde{\bY},\Gamma)$, the formulas (\ref{EQ:equality1}) and (\ref{EQ:subset}) yield that
$$\B(\bX,\tilde{\Gamma})=p_{\tilde{\bY}/\bX}(\B(\tilde{\bY},\tilde{\Gamma}))\subset p_{\tilde{\bY}/\bX}(\B(\tilde{\bY},\Gamma))\subset p_{\bY/\bX}(U).$$
\end{proof}

\begin{remark}
\label{RM:basicopensubsets}
We should remark that the set $\B(\bX,\Gamma)$ given in Definition~\ref{DEF:basicopensets} is nothing else but a subset of a basic gr-open subset of $\Spec(k[\bX])$ introduced by Hrushovski in \cite{hrushovski}. Let $G$ be a commutative algebraic group scheme over $k[\bX]$ and $\Gamma$ a finitely generated subgroup of $G(k[\bX])$. The set of primes $\frakp\in \Spec(k[\bX])$ satisfying that the canonical map $k[\bX]\rightarrow k[\bX]/\frakp$ is injective on $\Gamma$ is called a basic gr-open subset of $\Spec(k[\bX])$, denoted by $\rW(G,\Gamma)$. When $G=\bG_a$ or $G=\bG_m$, one has that
$$\B(\bX,\Gamma)=\rW(G,\Gamma)\cap \max(k[\bX])$$
where $\max(k[\bX])$ denotes the set of maximal ideals of $k[\bX]$. Hrushovski proved that if $k$ is a number field and $\dim \bX=1$ then $\rW(G,\Gamma)$ is infinite (see Lemma 5A.10 of \cite{hrushovski}). The key idea of his proof is reducing $G$ to the cases that $G$ is an Abelian variety or $\bG_m$ or $\bG_a$. The case that $G$ is an Abelian variety is due to N\'{e}ron (see for example Section 6 in Chapter 9 of \cite{lang} or Section 11.1 of \cite{serre}). The case when $G=\bG_a$ was proved in Lemma 5A.4 of \cite{hrushovski}. For the case when $G=\bG_m$, Hrushovski claimed that one can use an entirely similar argument as that in the proof of N\'{e}ron's Theorem. A similar claim was also made by Serre in Section 11.1 of \cite{serre} for the case when $k$ is a number field and $k(\bX)$ is a purely transcendental extension of $k$. To be complete, we shall provide a detailed proof for the case when $G=\bG_m$. Moreover we remove the restrictions on $k$ and $k[\bX]$.
\end{remark}

Now we turn to showing that basic open subsets of $\bX$ are not empty. We first  show that $\B(\bX,\Gamma)$ is not empty. From (\ref{eqn:basicopensets}), it suffices to prove that the set $\{\bfc\in \bY|\mbox{$v_\bfc$ is injective on $\Gamma$}\}$ is not empty. Furthermore since $\Gamma\subset k[\bY]$, one sees that $\B(\bY,\Gamma)=\{\bfc\in \bY|\mbox{$v_\bfc$ is injective on $\Gamma$}\}$. So it suffices to prove that $\B(\bY,\Gamma)\neq \emptyset$. Due to Noetherian normalization lemma, it is reasonable to make the following assumption.
\begin{convention}
Suppose that $\bY\subset k^m$ and denote
$k[\bY]=k[\eta_1,\cdots,\eta_m]$, where $\eta_1,\cdots,\eta_l\in L$ are algebraically independent over $k$ and $\eta_{l+i}\in L$ is integral over $k[\eta_1,\cdots,\eta_l]$. Set $\bfeta=(\eta_1,\cdots,\eta_m)$ and $ \bfeta_l=(\eta_1,\cdots,\eta_l).$
\end{convention}
To prove $\B(\bY,\Gamma)\neq \emptyset$, we need a generalization of Hilbert sets (see Section 12.1 of \cite{fried-jarden}).
Let $\tilde{k}\subset k$ be a field finitely generated over $\CQ$ such that the minimal polynomial of $\eta_{l+i}$ over $k(\bfeta_l)$ has coefficients in $\tilde{k}[\bfeta_l]$ for all $i=1,\cdots,m-l$.
Assume that $\bff$ is a finite set of polynomials in $\tilde{k}[\bfeta,z]$ irreducible over $\tilde{k}(\bfeta)$ and monic in $z$. Suppose that $g \in \tilde{k}[\bfeta]\setminus\{0\}$, and $\bfd=(d_1,\cdots,d_l)\in \bZ^l$ with positive $d_i$.
\begin{notation}
\label{DEF:hilbertian}
 $\H_{\tilde{k},\bY}(\bfd,\bff,g)$ denotes the set of $\bfc=(c_1,\cdots,c_m) \in \bY$ satisfying that
\begin{itemize}
\item [$(1)$]
    for $1\leq i \leq l$, $[\tilde{k}(c_1,\cdots,c_i):\tilde{k}(c_1,\cdots,c_{i-1})]\geq d_i$, and
\item [$(2)$]
    $g(\bfc)\neq 0$, and
\item [$(3)$]
   for each $f\in \bff$, $f(\bfc,z)$ is irreducible over $\tilde{k}(\bfc)$.
\end{itemize}
We call such $\H_{\tilde{k},\bY}(\bfd,\bff,g)$ a $\tilde{k}$-Hilbert set of $\bY$.
\end{notation}

Assume that $K$ is a field of characteristic zero, $\bfT=\{T_1,\cdots,T_m\}$ and $\bfY=\{y_1,\cdots,y_n\}$. For $g\in K[\bfT]\setminus \{0\}$ and $h_1,\dots,h_s\in K(\bfT)[\bfY]$ irreducible over $K(\bfT)$, denote by $H_K(h_1,\dots,h_s;g)$ the set of all $\bfc\in K^m$ with $g(\bfc)\neq 0$ and $h_1(\bfc,\bfY), \cdots, h_s(\bfc,\bfY)$ defined and irreducible in $K[\bfY]$.  In Section 12.1 of \cite{fried-jarden}, a set of the form $H_K(h_1,\dots,h_s;g)$ is called a Hilbert subset of $K^m$ and the field $K$ is called a hilbertian field if for every positive integer $m$, each Hilbert subset of $K^m$ is nonempty. One sees that $\H_{\tilde{k},\bY}((1,\cdots,1),\emptyset,g)=\bY_g$ and if $\bY=\tilde{k}^l$ then $\H_{\tilde{k},\bY}((1,\cdots,1),\bff,g)$ is a usual Hilbert set. Furthermore, one can easily verify that
$$\H_{\tilde{k},\bY}(\bfd_1,\bff_1,g_1)\cap \H_{\tilde{k},\bY}(\bfd_2,\bff_2,g_2)=\H_{\tilde{k},\bY}(\bar{\bfd},\bff_1\cup\bff_2,g_1g_2)$$
 where the $i$-th coordinate of $\bar{\bfd}$ is equal to the maximum of the $i$-th coordinates of $\bfd_1$ and $\bfd_2$ for all $i=1,\cdots,l$. From this, one sees that the intersection of finitely many $\tilde{k}$-Hilbert sets is a $\tilde{k}$-Hilbert set. Remark that if $\tilde{k}$ is replaced by an Omega-free PAC field $K$, Lemma 27.2.1 on page 660 of \cite{fried-jarden} implies that $\H_{K,\bY}((1,\cdots,1),\bff,g)$ is not empty.  We shall prove that every $\tilde{k}$-Hilbert set is nonempty.
\begin{lemma}
\label{LM:algextension}
Assume that $K$ is a hilbertian field and $\tilde{K}$ is a finite extension of $K$. For any positive integer $d$, there is $\alpha$ algebraic over $K$ satisfying that $[K(\alpha):K]=d$ and $K(\alpha)\cap \tilde{K}=K$.
\end{lemma}
\begin{proof}
Consider the polynomial $z^d-t \in K[z,t]$ which is irreducible over $\tilde{K}(t)$. Since $K$ is hilbertian, there is $c\in K$ such that $z^d-c$ is irreducible in $\tilde{K}[z]$ by Corollary 1.8 on page 10 of \cite{volklein}. Let $\alpha$ be a root of $z^d-c=0$ in $\bar{K}$. Then $[K(\alpha):K]=d$. If $K(\alpha)\cap \tilde{K}\neq K$ then $[K(\alpha):K(\alpha)\cap \tilde{K}]<d$. This implies that $z^d-c$ is reducible over $\tilde{K}$, a contradiction. Hence $K(\alpha)\cap \tilde{K}=K$.
\end{proof}

\begin{proposition}
\label{PROP:hilbertianset}
$\H_{\tilde{k},\bY}(\bfd,\bff,g)\neq \emptyset$.
\end{proposition}
\begin{proof}
Suppose that $\bff=\{f_1,\cdots,f_s\}$. For each $i=1,\cdots,s$, let $\alpha_i\in L$ satisfy that $f_i(\alpha_i)=0$ and let $\beta_i\in L$ be such that $\tilde{k}(\bfeta,\alpha_i)=\tilde{k}(\bfeta_l,\beta_i)$. We may choose $\beta_i$ to be integral over $\tilde{k}[\bfeta_l]$. Let $\tilde{f}_i$ be the polynomial in $\tilde{k}[y_1,\cdots,y_l,z]$ irreducible over $\tilde{k}$ and monic in $z$ such that $\tilde{f}_i(\bfeta_l,\beta_i)=0$. Then
\begin{align*}
\deg_z(\tilde{f}_i)&=[\tilde{k}(\bfeta_l, \beta_i):\tilde{k}(\bfeta_l)]=[\tilde{k}(\bfeta_l, \beta_i):\tilde{k}(\bfeta)][\tilde{k}(\bfeta):\tilde{k}(\bfeta_l)]\\
&=[\tilde{k}(\bfeta, \alpha_i):\tilde{k}(\bfeta)][\tilde{k}(\bfeta):\tilde{k}(\bfeta_l)]=\deg_z(f_i)[\tilde{k}(\bfeta):\tilde{k}(\bfeta_l)].
\end{align*}
Assume that $\beta_i=h_i(\bfeta, \alpha_i)/r(\bfeta_l)$ where $h_i\in \tilde{k}[y_1,\cdots,y_m,z]$ and $r\in \tilde{k}[y_1,\cdots,y_l]$. Let $k'$ be a finite extension of $\tilde{k}$ such that all factors of the $\tilde{f}_i$ irreducible over $k'$ are absolutely irreducible. Using Lemma~\ref{LM:algextension} repeatedly, we have $c_1,\cdots,c_l\in k$ such that $\tilde{k}(c_1,\cdots,c_l)\cap k'=\tilde{k}$ and for each $j=1,\cdots,l$,
$$
   [\tilde{k}(c_1,\cdots,c_j):\tilde{k}(c_1,\cdots,c_{j-1})]=d_j.
$$
 Write $\bfc_l=(c_1,\cdots,c_l)$. We claim that all $\tilde{f}_i$ are irreducible over $\tilde{k}(\bfc_l)$. Otherwise, assume that $\tilde{f}_i$ is reducible over $\tilde{k}(\bfc_l)$ for some $i$ and $q$ is one of its irreducible factor. Then $q$ is the product of some irreducible factors of $\tilde{f}_i$ in $k'[y_1,\cdots,y_l,z]$. Therefore the coefficients of $q$ are all in $k'\cap \tilde{k}(\bfc_l)$, i.e. $q\in \tilde{k}[y_1,\cdots,y_l,z]$. This contradicts the irreducibility of $\tilde{f}_i$. This proves our claim. It is easy to see that all $\tilde{f}_i(y_1+c_1,\cdots,y_l+c_l,z)$ are also irreducible over $\tilde{k}(\bfc_l)$. As $\tilde{k}(\bfc_l)$ is a finite extension of the hilbertian field $\tilde{k}$, by Lemma 12.2.2 on page 224 of \cite{fried-jarden}, there is a Hilbert set $H\subset \tilde{k}^l$ such that for each $\bfa\in H$, all $\tilde{f}_i(\bfa+\bfc_l,z)$ are irreducible over $\tilde{k}(\bfc_l)$. Let $\tilde{g}$ be the norm of $g$ down to $\tilde{k}(\bfeta_l)$. Since $g$ is integral over $\tilde{k}[\bfeta_l]$ and $\tilde{k}[\bfeta_l]$ is integrally closed, $\tilde{g}\in \tilde{k}[\bfeta_l]$. One sees that for any $\bfc\in \bY$ if $\tilde{g}(\bfc)\neq 0$ then $g(\bfc)\neq 0$. Let $\tilde{H}$ be the set of $\bfa\in H$ satisfying that $\tilde{g}(\bfa+\bfc_l)r(\bfa+\bfc_l)\neq 0$. Then $\tilde{H}\neq \emptyset$ as $H$ is Zariski dense. Now let $\bfb=(b_1,\cdots,b_m)\in \bY$ satisfy that $(b_1,\cdots,b_l)=\bfa+\bfc_l$ for some $\bfa\in \tilde{H}$. Such $\bfb$ exists because $\eta_{l+1},\cdots,\eta_m$ are integral over $k[\bfeta_l]$. Let $\bar{\alpha}_i\in k$ be a zero of $f_i(\bfb,z)$. Set $\bar{\beta}_i=h_i(\bfb,\bar{\alpha}_i)/r(\bfa+\bfc_l)$. Then $\bar{\beta}_i$ is a zero of $\tilde{f}_i(\bfa+\bfc_l,z)$ and since $\tilde{f}_i(\bfa+\bfc_l,z)$ is irreducible over $\tilde{k}(\bfc_l)$,
\begin{align*}
   \deg_z(\tilde{f}_i(\bfa+\bfc_l,z))&=[\tilde{k}(\bfa+\bfc_l,\bar{\beta}_i):\tilde{k}(\bfa+\bfc_l)]=[\tilde{k}(\bfc_l,\bar{\beta}_i):\tilde{k}(\bfc_l)]\\
   &\leq [\tilde{k}(\bfb,\bar{\alpha}_i):\tilde{k}(\bfc_l)]=[\tilde{k}(\bfb,\bar{\alpha}_i):\tilde{k}(\bfb)][\tilde{k}(\bfb):\tilde{k}(\bfc_l)]\\
   &\leq \deg_z(f_i(\bfb,z))[\tilde{k}(\bfeta):\tilde{k}(\bfeta_l)]=\deg_z(\tilde{f}_i).
\end{align*}
The last inequality holds because $\bar{\alpha}_i$ is a zero of $f_i(\bfb,z)$ and $\bfb\in \bY$ with $\bfeta$ as a generic point. At the same time, because $\deg_z(\tilde{f}_i(\bfa+\bfc_l,z))=\deg_z(\tilde{f}_i)$, one has that $\deg_z(f_i(\bfb,z))=[\tilde{k}(\bfb,\bar{\alpha}_i):\tilde{k}(\bfb)]$. This implies that $f_i(\bfb,z)$ is irreducible over $\tilde{k}(\bfb)$. It is obvious that for each $j=1,\cdots,l$,
$$
    [\tilde{k}(b_1,\cdots,b_j):\tilde{k}(b_1,\cdots,b_{j-1})]= [\tilde{k}(c_1,\cdots,c_j):\tilde{k}(c_1,\cdots,c_{j-1})]\geq d_j.
$$
Therefore $\bfb\in \H_{\tilde{k},\bY}(\bfd,\bff,g)$.
\end{proof}
\begin{corollary}
\label{COR:rationalsols}
Suppose that $h\in \tilde{k}[\bfeta][z]$ is monic  and of degree $\geq 1$ in $z$. Then there exists a $\tilde{k}$-Hilbert set $V$ of $\bY$ such that for any $\bfc\in V$, $h(\bfc,z)=0$ has a root in $\tilde{k}(\bfc)$ if and only if $h=0$ has a root in $\tilde{k}(\bfeta)$.
\end{corollary}
\begin{proof}
Decompose $h$ into irreducible polynomials in $\tilde{k}(\bfeta)[z]$, say $h_1,h_2,\cdots,h_s$. Pick a suitable nonzero $g\in \tilde{k}[\bfeta]$ such that for each $i=1,\cdots,s$, $g^{\deg_z(h_i)}h_i=f_i(gz)$ for some $f_i\in \tilde{k}[\bfeta,z]$ being monic in $z$. One sees that $f_i$ is irreducible over $\tilde{k}(\bfeta)$ and moreover $h_i$ has a zero in $\tilde{k}(\bfeta)$ if and only if so does $f_i$. Let $\bff=\{f_1,\cdots,f_s\}$ and $V= \H_{\tilde{k},\bY}((1,\cdots,1),\bff,g)$. Suppose that $\bfc\in V$. Then $f_i(\bfc,z)$ has a zero in $\tilde{k}(\bfc)$ if and only if so does $h_i(\bfc,z)$. For an irreducible polynomial in $z$, it has a zero in its coefficient field if and only if it is of degree one. The corollary then follows from the fact that $h_i(\bfc,z)$ is irreducible and
$$\deg_z(h_i(\bfc,z))=\deg_z(f_i(\bfc,z))=\deg_z(f_i)=\deg_z(h_i).$$
\end{proof}
Due to Lemma 5A.3 and Remark 5A.3R of \cite{hrushovski}, one has the following proposition.
\begin{proposition}
\label{PROP:additivegroups}
Suppose that $\Gamma$ is a finitely generated subgroup of $\bG_a(\tilde{k}[\bfeta])$. Then there is an $l$-tuple of positive integers $\bfd$ such that $\H_{\tilde{k},\bY}(\bfd,\emptyset,1)\subset \B(\bY,\Gamma)$.
\end{proposition}
\begin{proof}
 We have that $\{\eta_1,\cdots,\eta_l\}$ is a transcendental basis of $\overline{k(\bY)}/k$. Let $V$ be the $\tilde{k}$-vector space in $\tilde{k}[\bfeta]$ spanned by $\Gamma$. As $\Gamma$ is finitely generated, $V$ is of finite dimension. By Remark 5A.3R of \cite{hrushovski}, there are positive integers $d_1,\cdots,d_l$ such that for any $\tilde{k}$-homomorphism $h: \tilde{k}[\bfeta]\rightarrow \tilde{k}^a\subset k$, if
 \begin{equation}
 \label{EQ:inequalities}
 [\tilde{k}(h(\eta_1),\dots, h(\eta_i)):\tilde{k}(h(\eta_1),\dots, h(\eta_{i-1}))]\geq d_i
 \end{equation}
for every $i=1,\cdots,l$,  then $h$ is injective on $V$. Here $\tilde{k}^a$ denotes the algebraic closure of $\tilde{k}$. Now let $\bfc\in \H_{\tilde{k},\bY}((d_1,\cdots,d_l),\emptyset,1)$. Then the restriction of $v_\bfc$ on $\tilde{k}[\bfeta]$ is a $\tilde{k}$-homomorphism from $\tilde{k}[\bfeta]$ to $k$ and $v_\bfc(\eta_i)=c_i$, where $\bfc=(c_1,\cdots,c_m)$. By definition (see Notation~\ref{DEF:hilbertian} (1)), 
\begin{align*}
  [\tilde{k}(v_\bfc(\eta_1),\cdots,v_\bfc(\eta_i)):\tilde{k}(v_\bfc(\eta_1),\cdots,v_\bfc(\eta_{i-1}))]&=[\tilde{k}(c_1,\cdots,c_i):\tilde{k}(c_1,\cdots,c_{i-1})]\\
  &\geq d_i
\end{align*}
i.e. the restriction of $v_\bfc$ satisfies the conditions (\ref{EQ:inequalities}). The above statement following from Remark 5A.3R of \cite{hrushovski} then implies that $v_\bfc$ is injective on $V$ and thus on $\Gamma$. In other words, $\bfc\in \B(\bY,\Gamma)$.
\end{proof}

Next, we are going to deal with the case that $\Gamma$ is a finitely generated subgroup of $\bG_m(\tilde{k}[\bfeta])$. It has been claimed on page 154 of \cite{serre} and in Discussion 5A.8 (4) of \cite{hrushovski} that the proof of N\'{e}ron's theorem can be applied to proving that $\B(\bY,\Gamma)\neq \emptyset$. The readers are referred to Section 6 in Chapter 9 of \cite{lang} or Section 11.1 of \cite{serre} for the proof of N\'{e}ron's theorem. Here we present a detailed proof of the claim made by Hrushovski and Serre. Let $K\subset L$ be a subfield.
\begin{definition}
Suppose $\Gamma$ is a subgroup of $\bG_m(K)$. The radical of $\Gamma$ in $K$, denoted by $\rad_K(\Gamma)$, is defined to be
$$
    \{\alpha\in \bG_m(K) \,\,| \,\,\mbox{$\exists\,\,l>0$ s.t. $\alpha^l \in \Gamma$}\}.
$$
We say $\Gamma$ is radical in $K$ if $\Gamma=\rad_K(\Gamma)$.
\end{definition}
It is easy to see that $\rad_K(\Gamma)$ is also a subgroup of $\bG_m(K)$. Moreover, we have the following proposition.
\begin{proposition}
\label{PROP:groupextension}
Suppose that $K$ is a field finitely generated over $\CQ$ and $\Gamma$ is a finitely generated subgroup of $\bG_m(K)$. Then $\rad_K(\Gamma)$ is also finitely generated.
\end{proposition}
\begin{proof}
Assume that $a_1,\cdots,a_m$ are generators of $\Gamma$.
We first prove the case that $K$ is a number field. Let $\mathfrak{p}_1,\cdots,\mathfrak{p}_\ell$ be all prime ideals of $\mathcal{O}_K$ satisfying that for each $1\leq i\leq \ell$, $\ord_{\mathfrak{p}_i}(a_j)\neq 0$ for some $1\leq j\leq m$, where $\ord_{\frakp_i}(a_j)$ denotes the order of $a_j$ at $\frakp_i$. Consider the group homomorphism $\varphi: \rad_K(\Gamma) \rightarrow \bZ^\ell$ defined by
  $$
      \varphi(\alpha)=(\ord_{\mathfrak{p}_1}(\alpha),\cdots,\ord_{\mathfrak{p}_\ell}(\alpha)).
  $$
One can verify that $\ker(\varphi)=\rad_K(\Gamma)\cap \mathcal{O}_K^\times$ and so the kernel is finitely generated, because $\mathcal{O}_K^\times$ is finitely generated. The image of $\varphi$ is also finitely generated, as it is a subgroup of $\bZ^\ell$. Hence $\rad_K(\Gamma)$ is finitely generated.

 Now assume that $K$ is transcendental over $\CQ$.  Due to the results on page 99 of \cite{zariski-samuel}, there is a set $S^\star$ of prime divisors of $K/\CQ$ such that for any $b\in K$ if $\ord_{\mathfrak{p}}(b)\geq 0$ for all $\mathfrak{p}\in S^\star$ then $b$ is algebraic over $\CQ$. Let $\mathfrak{p}_1,\cdots,\mathfrak{p}_\ell$ be all elements in $S^\star$ satisfying that for each $1\leq i\leq \ell$, $\ord_{\mathfrak{p}_i}(a_j)\neq 0$ for some $1\leq j\leq m$. Similarly, consider the group homomorphism $\psi: \rad_K(\Gamma) \rightarrow \bZ^\ell$ defined by
  $$
      \psi(\alpha)=(\ord_{\mathfrak{p}_1}(\alpha),\cdots,\ord_{\mathfrak{p}_\ell}(\alpha)).
  $$
One can check that $\ker(\psi)=\tilde{\CQ}\cap \rad_K(\Gamma)$ where $\tilde{\CQ}$ is the algebraic closure of $\CQ$ in $K$. The image of $\psi$ is a subgroup of $\bZ^\ell$ and so it is finitely generated. Therefore to show that $\rad_K(\Gamma)$ is finitely generated, it suffices to show that $\ker(\psi)$ is finitely generated. Let $R=\tilde{\CQ}[a_1,1/a_1,\cdots,a_m,1/a_m]$ and let $\phi$ be a $\tilde{\CQ}$-homomorphism from $R$ to $\bQ$. Then $\phi(a_i)\neq 0$ for all $1\leq i \leq m$. Let $\tilde{\Gamma}$ be the subgroup of $\bG_m(\bQ)$ generated by $\phi(a_1),\cdots,\phi(a_m)$ and let $E=\tilde{\CQ}(\phi(a_1),\cdots,\phi(a_m))$. Then $\tilde{\Gamma}=\phi(\Gamma)$ and $E$ is a number field. Suppose that $\gamma\in \ker(\psi)$, i.e.$\gamma\in \tilde{\CQ}$ and $\gamma^d\in \Gamma$ for some $d>0$. Applying $\phi$ to $\gamma$ yields that
$
   \gamma^d=\phi(\gamma)^d\in \tilde{\Gamma}.
$
This implies that $\gamma\in \rad_E(\tilde{\Gamma})$ and thus $\ker(\psi)\subset \rad_E(\tilde{\Gamma})$. Since $E$ is a number field, $\rad_E(\tilde{\Gamma})$ is finitely generated as we have already proved. So $\ker(\psi)$ is finitely generated.
\end{proof}
The example below shows that if $K$ is not finitely generated over $\CQ$ then $\rad_K(\Gamma)$ may not be finitely generated.
\begin{example}
Let $K=\CQ(\xi_2,\xi_3,\cdots)$ where $\xi_i$ is a primitive $i$-th root of unity, and let $\Gamma=\{1\}$. Then $\rad_K(\Gamma)$ contains all $\xi_i$, and thus it is not finitely generated.
\end{example}

 For a positive integer $\ell$ and a subgroup $\Gamma$ of $\bG_m(\tilde{k}[\bfeta])$, denote
 $$
    \Gamma_\ell=\{\gamma\in \Gamma | \gamma^\ell=1\}.
 $$
\begin{lemma}
\label{LM:torsiongroup}
Suppose that $\ell$ is a positive integer and $\Gamma$ is a finitely generated subgroup of $\bG_m(\tilde{k}[\bfeta])$ which is radical in $\tilde{k}(\bfeta)$. Then there exists a $\tilde{k}$-Hilbert set $V$ of $\bY$ such that for any $\bfc\in V$, $v_\bfc(\Gamma)$ is a  subgroup of $\bG_m(\tilde{k}(\bfc))$ and $v_\bfc(\Gamma_\ell)=v_\bfc(\Gamma)_\ell$. Moreover $v_\bfc(\Gamma)$ is finitely generated.
\end{lemma}
\begin{proof}
   Let $h$ be a polynomial in $\tilde{k}[\bfeta][z]$ such that
   $$z^\ell-1=h\prod_{c\in \Gamma_\ell}(z-c).$$
Then $h=0$ has no roots in $\tilde{k}(\bfeta)$, because $\Gamma$ is radical in $\tilde{k}(\bfeta)$. By Corollary~\ref{COR:rationalsols}, there exists a $\tilde{k}$-Hilbert set $\tilde{V}$ of $\bY$ such that for any $\bfc\in \tilde{V}$, $h(\bfc,z)=0$ has no root in $\tilde{k}(\bfc)$. Set $g=b_1\cdots b_N$ where $b_1,\cdots,b_N$ are generators of $\Gamma$. Let $V=\tilde{V}\cap \bY_g$ and $\bfc\in V$. Then $b_i(\bfc)\neq 0$ for all $1\leq i \leq N$ and thus the restriction of $v_\bfc$ on $\Gamma$ is a group homomorphism. This implies that $v_\bfc(\Gamma)$ is a finitely generated subgroup of $\bG_m(\tilde{k}(\bfc))$ because $\Gamma$ is finitely generated. In addition, note that $\bfc\in \tilde{V}$ and 
$$z^\ell-1=h(\bfc,z)\prod_{c\in \Gamma_\ell}(z-v_\bfc(c)).$$
One sees that  $v_\bfc(\Gamma)_\ell$, the set of all roots of $z^\ell-1=0$ in $v_\bfc(\Gamma)$, equals $\{v_\bfc(c)|c\in \Gamma_\ell\}$ and the latter set is nothing else but $v_\bfc(\Gamma_\ell)$.
\end{proof}

\begin{proposition}
\label{PROP:multiplicativegroups}
Suppose that $\Gamma$ is a finitely generated subgroup of $\bG_m(\tilde{k}[\bfeta])$. There exists a $\tilde{k}$-Hilbert set $V$ of $\bY$ such that $V\subset \B(\bY,\Gamma)$.
\end{proposition}
\begin{proof}
Set $\tilde{\Gamma}=\rad_{\tilde{k}(\bfeta)}(\Gamma)$. Then due to Proposition~\ref{PROP:groupextension}, $\tilde{\Gamma}$ is finitely generated. Let $q\in \tilde{k}[\bfeta_l]$ be nonzero element such that $\tilde{\Gamma}\subset \tilde{k}[\bfeta,1/q]$. We will first show the proposition for $\bY_q$ and $\tilde{\Gamma}$.

Let $T$ be the torsion group of $\tilde{\Gamma}$ and $\ell$ an integer greater than $1$ and divided by $|T|$. By Lemma~\ref{LM:torsiongroup}, there exists a $\tilde{k}$-Hilbert set $V_1$ of $\bY_q$ such that for any $\bfa\in V_1$, $v_{\bfa}(\tilde{\Gamma})$ is a finitely generated subgroup of $\bG_m(\tilde{k}(\bfa))$ and $v_{\bfa}(\tilde{\Gamma}_\ell)=v_{\bfa}(\tilde{\Gamma})_\ell$. Suppose that $\{b_1=1,b_2,\cdots,b_\nu\}$ is a set of representatives of $\tilde{\Gamma}/\tilde{\Gamma}^\ell$. Corollary~\ref{COR:rationalsols} implies that there exists a $\tilde{k}$-Hilbert set $V_2$ of $\bY_q$ such that for any $\bfa\in V_2$, $z^\ell-v_{\bfa}(b_i)=0$ has a root in $\tilde{k}(\bfa)$ if and only if $z^\ell-b_i=0$ has a root in $\tilde{k}(\bfeta)$.  Since $\tilde{\Gamma}$ is radical in $\tilde{k}(\bfeta)$, all roots of $z^\ell-b_i=0$ in $\tilde{k}(\bfeta)$ are in $\tilde{\Gamma}$ and then $z^\ell-b_i=0$ has a root in $\tilde{k}(\bfeta)$ only if $i=1$. Thus for each $\bfa\in V_2$, $z^\ell-v_{\bfa}(b_i)=0$ has a root in $\tilde{k}(\bfa)$ only if $i=1$.  We claim that $V_1\cap V_2\subset \B(\bY_q,\tilde{\Gamma})$. Suppose that $\bfa\in V_1\cap V_2$. Let $I=v_{\bfa}^{-1}(1)\cap \tilde{\Gamma}$. Then $I$ is a finitely generated subgroup of $\tilde{\Gamma}$. We shall show that $I=I^\ell$ and $I$ is free. This will imply $I=1$ because $\ell>1$, and thus $v_{\bfa}$ is injective on $\tilde{\Gamma}$ i.e. $\bfa\in \B(\bY_q,\tilde{\Gamma})$. Since $|T|$ divides $\ell$, if $I=I^\ell$ then $I$ is torsion-free and then it is free. So we only need to prove that $I=I^\ell$. Suppose $w\in I$. Write $w=b_i \bar{w}^\ell$ for some $i$ and some $\bar{w}\in \tilde{\Gamma}$. Then $v_{\bfa}(\bar{w})^{-\ell}=v_{\bfa}(b_i)$. In other words, $v_{\bfa}(\bar{w})^{-1}$ is a root of $z^\ell-v_{\bfa}(b_i)=0$ in $\tilde{k}(\bfa)$. The assumption on $\bfa$ indicates that $b_i=1$. This implies $w=\bar{w}^\ell$ and then $v_{\bfa}(\bar{w})^\ell=1$, i.e. $v_{\bfa}(\bar{w})\in v_{\bfa}(\tilde{\Gamma})_\ell$. As $v_{\bfa}(\tilde{\Gamma}_\ell)=v_{\bfa}(\tilde{\Gamma})_\ell$, there is $u\in \tilde{\Gamma}_\ell$ such that $v_{\bfa}(\bar{w})=v_{\bfa}(u)$. For such $u$, $\bar{w} u^{-1}\in I$. As $u^\ell=1$, $w=\bar{w}^\ell=(\bar{w}u^{-1})^\ell\in I^\ell$. Therefore $I=I^\ell$.  This proves our claim.

Now assume that $V_1\cap V_2=\H_{\tilde{k},\bY_q}(\bfd,\{\tilde{f}_1,\cdots,\tilde{f}_s\},\tilde{g})$ where $\tilde{f}_i \in \tilde{k}[\bfeta,1/q][z]$ irreducible over $\tilde{k}(\bfeta)$ and monic in $z$, and $\tilde{g}\in \tilde{k}[\bfeta,1/q]$. For each $i=1,\cdots,s$, there are positive integers $d_i,e_i$ and $f_i\in \tilde{k}[\bfeta,z]$ irreducible over $\tilde{k}(\bfeta)$ and monic in $z$ such that $q^{d_i}\tilde{f}_i=f_i(q^{e_i}z)$. Set
$
   \bff=\{f_1,\cdots,f_s\}.
$
Let $\mu$ be a positive integer such that $q^\mu \tilde{g}\in \tilde{k}[\bfeta]$ and set $g=q^{\mu+1}\tilde{g}$. One then has that $\H_{\tilde{k},\bY}(\bfd,\bff,g)\subset \B(\bY,\Gamma)$.
\end{proof}
\begin{theorem}
\label{TH:basicopen}
Every basic open subset of $\bX$ is Zariski dense.
\end{theorem}
\begin{proof}
We first show that every basic open subset of $\bX$ is not empty.
Assume that $\Gamma_1,\cdots,\Gamma_s$ are subgroups of $\bG_a(L)$ and $\Gamma_{s+1},\cdots,\Gamma_\ell$ are subgroups of $\bG_m(L)$. Let $\bY_i$ and $\tilde{\bY}$ be the varieties associated to $k[\bfeta,\Gamma_i]$ and $k[\bfeta, \cup_{i=1}^\ell \Gamma_i]$ respectively. By definition, one has that $\B(\bX,\Gamma_i)=p_{\bY_i/\bX}(\B(\bY_i,\Gamma_i))$ and
$p_{\tilde{\bY}/\bY_i}(\B(\tilde{\bY},\Gamma_i))\subset \B(\bY_i,\Gamma_i).$ Applying $p_{\bY_i/\bX}$ to the latter inclusion yields that
$$
   p_{\bY_i/\bX}(p_{\tilde{\bY}/\bY_i}(\B(\tilde{\bY},\Gamma_i)))\subset p_{\bY_i/\bX}(\B(\bY_i,\Gamma_i))=\B(\bX,\Gamma_i).
$$
Therefore to show that $\cap_{i=1}^\ell \B(\bX,\Gamma_i)\neq \emptyset$, it suffices to show that $\cap_{i=1}^\ell \B(\tilde{\bY},\Gamma_i)\neq \emptyset$. The latter assertion follows from Propositions~\ref{PROP:additivegroups} and \ref{PROP:multiplicativegroups} where $\tilde{k}$ is taken to be the field finitely generated over $\CQ$ such that the $\Gamma_i$ are in $\tilde{k}[\bfeta]$.

Suppose that $U$ is a basic open subset of $\bX$ and $U$ is not Zariski dense, i.e. there is a nonzero $g\in k[\bX]$ which vanishes on $U$. By Definition~\ref{DEF:basicopensets}, $v_\bfc(g)=g(\bfc)\neq 0$ for all $\bfc\in \bX_g$. So $U\cap \bX_g=\emptyset$. However by Definition~\ref{DEF:basicopensets} $U\cap \bX_g$ is a basic open subset of $\bX$ and thus it is not empty, a contradiction.
\end{proof}
The following two lemmas will be used later.
\begin{lemma}
\label{LM:integers}
Suppose that $f\in k[\bX][z]$. There is a finitely generated subgroup $\Gamma$ of $\bG_a(\overline{k(\bX)})$ such that for any $\bfc\in \B(\bX,\Gamma)$, one has that $\Z(f)=\Z(f(\bfc,z))$.
\end{lemma}
\begin{proof}
  Let $\alpha_1,\cdots,\alpha_\ell$ be all zeroes of $f$ in $\overline{k(\bX)}\setminus \bZ$ and $a$ be the leading coefficient of $f$. Set $\Gamma$ to be the subgroup of $\bG_a(\overline{k(\bX)})$ generated by $1,a, \alpha_1,\cdots,\alpha_\ell$ and let $\bY$ be the variety associated to $k[\bX][\alpha_1,\cdots,\alpha_\ell]$. Suppose that $\bfc\in \B(\bX,\Gamma)$. By the definition of basic open subsets, $\bfc$ can be extended to a point $\tilde{\bfc}\in \B(\bY,\Gamma)$. One sees that the $\Z(f)\subset \Z(f(\tilde{\bfc},z))$ and $v_{\tilde{\bfc}}(\alpha_i)\notin \bZ$ for all $1\leq i \leq \ell$. Therefore $\Z(f)=\Z(f(\tilde{\bfc},z))=\Z(f(\bfc,z))$.
\end{proof}
In the following, for a matrix $M$ with entries in $k[\bX]$, the rank of $M$ is defined to be the rank of $M$ regarded as a matrix over $k(\bX)$.
\begin{lemma}
\label{LM:matrices}
Assume that $M$ is a matrix in $k[\bX]^{\ell \times n}$. Then there is a nonzero $g\in k[\bX]$ such that for any $\bfc\in \bX_g$, $\rank(M)=\rank(M(\bfc))$.
\end{lemma}
\begin{proof}
Clearly, $\rank(M(\bfc))\leq \rank(M)$ for all $\bfc\in \bX$. Let $r=\rank(M)$. If $r=0$, there is nothing to prove. Suppose that $r>0$ and $g$ is a nonzero $r\times r$ minors of $M$. Suppose $\bfc\in \bX_g$. It is easy to see that $v_\bfc(g)$ is a $r\times r$ minor of $M(\bfc)$. Since $v_\bfc(g)\neq 0$, $\rank(M(\bfc))\geq r$. This implies that $r=\rank(M(\bfc))$.
\end{proof}

\section{Algebraic groups under specialization}
\label{SEC:algebraicgroups}
Assume that $G$ is an algebraic subgroup of $\GL_n(\overline{k(\bX)})$ defined over $k[\bX]$, i.e. the vanishing ideal of $G$ is generated by a finite subset $S$ in $k[\bX][X]$. Let $\frakX$ be a basis of $\bfchi(G^\circ)$ as a free abelian group. We further assume that every character in $\frakX$ is represented by an element in $k[\bX][X,1/\det(X)]$. Remark that $G^\circ$ is also defined over $k[\bX]$ (see (7.3) on page 210 of \cite{humphreys}). We shall use $G_\bfc$ to denote the variety in $\GL_n(k)$ defined by $v_\bfc(S)$ for $\bfc\in \bX$. In this section, we shall prove that there is a nonzero $c\in k[\bX]$ such that if $\bfc\in \bX_c$ then $G_\bfc$ is an algebraic subgroup of $\GL_n(k)$ satisfying that $\dim(G_\bfc)=\dim(G)$ and $v_\bfc(\frakX)$ is a basis of $\bfchi(G_\bfc^\circ)$. Note that when $G$ is commutative, the results of Lemma~\ref{LM:unipotents} and Proposition~\ref{PROP:character} have already appeared in \cite{hrushovski} (see Example 5A.6 and Lemma 5.11 respectively).

Let us start with a few remarks which follows from the application of Remark 5A.5 of \cite{hrushovski} to polynomial equalities with coefficients in $k(\bX)$.
\begin{remark}
\label{RM:corollaries}
\begin{enumerate}
\item
If $\tilde{S}\subset k[\bX][X]$ is a finite set defining $G$, then there is a nonempty open subset $U$ of $\bX$ such that $G_\bfc$ is defined by $v_\bfc(\tilde{S})$ for all $\bfc\in U$. Thus the notation $G_\bfc$ makes sense. To see this, note that $\tilde{S}$ and $S$ define the same variety if and only if they generate the same radical ideal i.e. for every $P\in S, \tilde{P}\in \tilde{S}$, there are $\alpha_{P,\tilde{Q}}, \beta_{\tilde{P},Q}\in k(\bX)[X]$ such that
$$
    P^{d_P}=\sum_{\tilde{Q}\in \tilde{S} } \alpha_{P,\tilde{Q}}\tilde{Q}, \,\,\tilde{P}=\sum_{Q\in S} \beta_{\tilde{P}, Q}Q,
$$
where $d_P$ is a positive integer. Any nonempty open subset of $\bX$ on which all $\alpha_{P,\tilde{Q}}, \beta_{\tilde{P},Q}$ are well-defined will be the set as required. The open subsets in (2) and (3) below can be obtained similarly.
\item $G_\bfc$ is an algebraic group for all $\bfc$ being in some nonempty open subset of $\bX$. By Exercise 5 on page 57 of \cite{humphreys}, for a variety $H$ in $\GL_n(\overline{k(\bX)})$, $H$ is an algebraic group if and only if $I_n\in H$ and $H$ is closed under taking products. The latter condition can be described as follows: For each $P\in S$, there are $\alpha_{P,Q}, \beta_{P,Q}\in k(\bX)[X,Z,1/\det(XZ)]$ such that
$$
    P(XZ)=\sum_{Q\in S} \alpha_{P,Q}Q(X)+\sum_{Q\in S}\beta_{P,Q}Q(Z).
$$
Likewise, if $\chi$ is a character of $G$ then $v_\bfc(\chi)$ is a character of $G_\bfc$ for all $\bfc$ being in some nonempty open subset of $\bX$.
\item Suppose that $H$ and $\tilde{H}$ are two varieties defined over $k[\bX]$ and $H\cap \tilde{H}=\emptyset$. Then $H_\bfc\cap \tilde{H}_\bfc=\emptyset$ for all $\bfc$ being in some nonempty open subset of $\bX$. Note that $H\cap \tilde{H}=\emptyset$ if and only if there are polynomials $P$ and $Q$ in the vanishing ideal of $H$ and $\tilde{H}$ respectively such that $P+Q=1$.
\end{enumerate}
\end{remark}

We can view $G$ as a family of algebraic varieties $G_\bfc$ in $\GL_n(k)$ parameterized by $\bX$. More precisely, suppose that $\bX\subset k^m$ and $\bfeta$ is a generic point of $\bX$. Denote
$$
    J=\{P\in k[y_1,\cdots,y_m,X,1/\det(X)] | \,\forall\,\bfb\in G, P(\bfeta,\bfb)=0\}.
$$
Let $\bY\subset k^m\times \GL_n(k)$ be the variety defined by $J$. Then $\bY$ is a variety over $k$ of dimension $\dim(\bX)+\dim(G)$. Define
\[
  \begin{array}{ccccccccccccc}
  \pi_1:& \bY  &\longrightarrow &\bX &   & \pi_2:& \bY  &\longrightarrow &\GL_n(k)\\
        & (\bfc,\bfb)&\longrightarrow & \bfc & &    & (\bfc,\bfb)&\longrightarrow & \bfb
  \end{array}.
\]
One sees that $G=\pi_2(\pi_1^{-1}(\bfeta))$. Note that $\pi_2(\pi_1^{-1}(\bfc))$ is the variety in $\GL_n(k)$ defined by $\{P(\bfc,X,1/\det(X))|P\in J\}$.
\begin{proposition}
\label{PROP:groupsfibre} There is a nonempty open subset $U$ of $\bX$ such that for any $\bfc\in U$, $G_\bfc$ is an algebraic subgroup of $\GL_n(k)$ with dimension $\dim(G)$ and
$$
   [G_\bfc:G_\bfc^\circ]=[G:G^\circ]=\ell.
$$
\end{proposition}
\begin{proof}
Note that $\{P(\bfeta,X,1/\det(X))|P\in J\}$ also defines $G$. The discussion in Remark~\ref{RM:corollaries} (1) implies that there is a nonempty open subset $\tilde{U}$ of $\bX$ such that $G_\bfc=\pi_2(\pi_1^{-1}(\bfc))$ for any $\bfc\in \tilde{U}$.
Hence it suffices to prove the proposition for $\pi_2(\pi_1^{-1}(\bfc))$.
Let $G_1,\cdots, G_\ell$ be all irreducible component of $G$. Let $D$ be a finitely generated $k[\bX]$-algebra in $\overline{k(\bX)}$ such that each $G_i$ is defined over $D$, i.e. the vanishing ideal of each $G_i$ in $\overline{k(\bX)}[X,1/\det(X)]$ is generated by finitely many polynomials in $D[X,1/\det(X)]$. Let $\tilde{\bX}$ be the variety over $k$ associated to $D$. By Lemma~\ref{LM:kolchinlemma}, for each nonempty open subset $\tilde{V}$ of $\tilde{\bX}$, there is a nonempty open subset $V$ of $\bX$ such that $V\subset p_{\tilde{\bX}/\bX}(\tilde{V})$. Furthermore, as the morphism $p_{\tilde{\bX}/\bX}$ is induced by the inclusion $k[\bX]\subset k[\tilde{\bX}]$, one sees that $v_\bfc(\bfeta)=v_{p_{\tilde{\bX}/\bX}(\bfc)}(\bfeta)$ for all $\bfc\in \tilde{\bX}$. This implies that $G_\bfc=G_{p_{\tilde{\bX}/\bX}(\bfc)}$ for all $\bfc\in \tilde{\bX}$. Therefore it suffices to prove the proposition with the variety $\tilde{\bX}$ over whose coordinate ring all $G_i$ are defined. In the following, for the sake of notation, we assume that all $G_i$ are defined over $k[\bX]$. Let $\bfxi_i$ be a generic point of $G_i$ over $\overline{k(\bX)}$ and set
$$
   J_i=\{Q\in k[y_1,\cdots,y_m, X, 1/\det(X)] | Q(\bfeta, \bfxi_i)=0\}.
$$
Let $\bY_i$ be the variety over $k$ defined by $J_i$. Then $\bY_i$ is irreducible because it has a generic point $(\bfeta,\bfxi_i)$. Moreover one can verify that $J=\cap_{i=1}^\ell J_i$. Hence $\bY=\cup_{i=1}^\ell \bY_i$. Additionally, one has that $G_i=\pi_2(\pi_1|_{\bY_i}^{-1}(\bfeta))$ and $\pi_2(\pi_1^{-1}(\bfc))=\cup_{i=1}^\ell \pi_2(\pi_1|_{\bY_i}^{-1}(\bfc))$.

By Remark~\ref{RM:corollaries}, there is a nonempty open subset $U_1$ of $\bX$ such that $\pi_2(\pi_1^{-1}(\bfc))$ is an algebraic subgroup for any $\bfc\in U_1$.
Note that $\pi_1|_{\bY_i}$ is dominant and because $G_i$ is irreducible over $\overline{k(\bX)}$, so is $\pi_1|_{\bY_i}^{-1}(\bfeta)$ which is equal to $\bfeta \times G_i$. By Theorem 1 on page 139 of \cite{shafarevich} and Proposition on page 33 of \cite{humphreys}, there is a nonempty open subset $U_2 \subset \bX$ such that for any $\bfc\in U_2$, $\pi_1|_{\bY_i}^{-1}(\bfc)$ is irreducible and of dimension $\dim(G)$. Since $G_i\cap G_j=\emptyset$ if $i\neq j$, by Remark~\ref{RM:corollaries} again, there is a nonempty open subset $U_3$ of $\bX$ such that for any $\bfc\in U_3$ and $i\neq j$,
$
   \pi_2(\pi_1|_{\bY_i}^{-1}(\bfc))\cap \pi_2(\pi_1|_{\bY_j}^{-1}(\bfc))=\emptyset.
$
Now set $U=U_1\cap U_2\cap U_3$. Then for any $\bfc \in U$, we have that $\pi_2(\pi_1^{-1}(\bfc))$ is an algebraic group and
$
   [\pi_2(\pi_1^{-1}(\bfc)):\pi_2(\pi_1^{-1}(\bfc))^\circ]=[G:G^\circ]=\ell.
$ Finally, note that $\pi_1|_{\bY_i}^{-1}(\bfc)=\bfc\times \pi_2(\pi_1|_{\bY_i}^{-1}(\bfc))$. Hence for each $\bfc\in U$, $\pi_2(\pi_1|_{\bY_i}^{-1}(\bfc))$ is of dimension $\dim(G)$ and so is $\pi_2(\pi_1^{-1}(\bfc))$. 
\end{proof}

\begin{lemma}
\label{LM:unipotents}
Assume that $G$ is generated by unipotent elements. Then there is a nonempty open subset $U$ of $\bX$ such that for any $\bfc\in U$, $G_\bfc$ is an algebraic group generated by unipotent elements and of dimension $\dim(G)$.
\end{lemma}
\begin{proof}
Due to Lemma C on page 96 of \cite{humphreys}, any unipotent element of $G$ that is not equal to the identity generates a connected 1-dimensional algebraic subgroup of $G$. Let $\calU$ be the set of all connected 1-dimensional algebraic subgroups of $G$ and $\tilde{G}$ the algebraic subgroup of $G$ generated by $\cup_{M\in \calU} M $. Then $\tilde{G}=G$ and by the proposition on page 55 of \cite{humphreys}, there are $M_1,\cdots,M_\ell$ in $\calU$ such that $\tilde{G}=M_1M_2\cdots M_\ell$. Furthermore, $\ell$ can be taken to be not greater than $2\dim(G)$. Now for each $i=1,\cdots,\ell$, there is a nilpotent matrix $\bfn_i$ in $\Mat_n(\overline{k(\bX)})$ such that
$$
    M_i=\left\{\left.\sum_{j=0}^{n-1} \frac{\bfn_i^j c^j}{j!}\right| c\in \overline{k(\bX)} \right\}.
$$
Let $D$ be a finitely generated $k[\bX]$-algebra in $\overline{k(\bX)}$ such that all entries of each $\bfn_i$ are in $D$, and let $\tilde{\bX}$ be the variety over $k$ associated to $D$. By Lemma~\ref{LM:kolchinlemma} and an argument similar to that in the proof of Proposition~\ref{PROP:groupsfibre}, one only need to prove the lemma with $\tilde{\bX}$. For the sake of notation, we may assume that all $\bfn_i$ are in $\Mat_n(k[\bX])$.
Set
$$
   (P_{i,j})= \prod_{i=1}^\ell \left(\sum_{j=0}^{n-1}\frac{\bfn_i^j t_i^j}{j!}\right)\in \GL_n(k[\bX][t_1,\cdots,t_\ell])
$$
where $t_1,\cdots,t_\ell$ are indeterminates.
Then $(P_{i,j})$ is a generic point of $G$. Assume that $d=\dim(G)$ and $P_{i_1,j_1}, \cdots, P_{i_d,j_d}$ are algebraically independent over $k(\bX)$. We claim that there is a nonempty open subset of $\bX$ such that for any $\bfc$ in this set, $v_\bfc(P_{i_1,j_1}), \cdots, v_\bfc(P_{i_d,j_d})$ are algebraically independent over $k$. For $\bfc\in \bX$, denote by $I_\bfc$ the ideal generated by all $y_{i,j}-v_\bfc(P_{i,j})$ in $k[t_1,\cdots,t_\ell, y_{1,1},\cdots,y_{n,n}]$. Let $S_\bfc$ be the reduced Gr\"{o}bner basis of $I_\bfc$ with respect to a lexicographic ordering
where every $t_i$ is greater than every $y_{l,m}$ and every $y_{l,m}$ with $(l,m)\neq (i_s,j_s)$ for all $s=1,\cdots,d$ is greater than every $y_{i_s,j_s}$.
Then $v_\bfc(P_{i_1,j_1}), \cdots, v_\bfc(P_{i_d,j_d})$ are algebraically dependent over $k$ if and only if $S_\bfc$ contains at least one polynomial in $k[y_{i_1,j_1},\cdots,y_{i_d,j_d}]$. Moreover for every $Q\in S_\bfc\cap k[y_{i_1,j_1},\cdots,y_{i_d,j_d}]$ one has that $Q(v_\bfc(P_{i_1,j_1}), \cdots, v_\bfc(P_{i_d,j_d}))=0$.
By Corollary 8.3 of \cite{dube}, there is an integer $N$ only depending on $n, \ell$ such that for every $\bfc\in \bX$, the total degree of each polynomial in $S_\bfc$ is not greater than $N$.
These imply that if $v_\bfc(P_{i_1,j_1}), \cdots, v_\bfc(P_{i_d,j_d})$ are algebraically dependent over $k$ then there is a nonzero $Q_\bfc$ in $k[y_{i_1,j_1},\cdots,y_{i_d,j_d}]$ of total degree not greater than $N$ such that
$$
   Q_\bfc(v_\bfc(P_{i_1,j_1}), \cdots, v_\bfc(P_{i_d,j_d}))=0.
$$
Now for nonnegative integers $s_1,\cdots,s_d$ with $s_1+\cdots+s_d\leq N$, write
$$
    P_{i_1,j_1}^{s_1}\cdots P_{i_d,j_d}^{s_d}=\sum_{\bfmu=(\mu_1,\cdots,\mu_\ell)} c_{s_1,\cdots,s_d,\bfmu} t_1^{\mu_1}t_2^{\mu_2}\cdots t_\ell^{\mu_\ell}
$$
where $0\leq \mu_i\leq N(n-1)$ and $c_{s_1,\cdots,s_d,\bfmu}\in k[\bX]$. Let $C$ be the $\binom{N+d}{d}\times (N(n-1)+1)^\ell$ matrix formed by $c_{s_1,\cdots,s_d,\bfmu}$. Since $P_{i_1,j_1}, \cdots, P_{i_d,j_d}$ are algebraically independent, $C$ is of full rank $\binom{N+d}{d}$, i.e. there is a nonzero $\binom{N+d}{d}\times \binom{N+d}{d}$-minor $g$ of $C$. Suppose that $v_\bfc(P_{i_1,j_1}), \cdots, v_\bfc(P_{i_d,j_d})$ are algebraically dependent over $k$ for some $\bfc\in \bX_g$. The choice of $N$ implies that the left kernel of $v_\bfc(C)$ has a nonzero element. This contradicts the fact that $v_\bfc(C)$ is of full rank. Thus $v_\bfc(P_{i_1,j_1}), \cdots, v_\bfc(P_{i_d,j_d})$ are algebraically independent over $k$ for all $\bfc\in \bX_g$. This proves the claim. By Proposition~\ref{PROP:groupsfibre}, there is a nonempty open subset $U_1$ of $\bX$ such that for any $\bfc\in U_1$, $G_\bfc$ is a connected algebraic group of dimension $\dim(G)$. Set $U=U_1\cap \bX_g$. Then for any $\bfc\in U$, since $(v_\bfc(P_{i,j}))$ is obviously a point of $G_\bfc$, it is a generic point of $G_\bfc$. Hence $G_\bfc$ is generated by unipotent elements.
\end{proof}
Let $H$ be a connected algebraic subgroup of $\GL_n(\overline{k(\bX)})$. The following lemma gives a criterion for a finite subset $\frakX$ to be a basis of $\bfchi(H)$ as a free abelian group. We say $\frakX$ is multiplicatively independent if the equality $\prod_{\chi\in \frakX}\chi^{d_\chi}=1$ with $d_\chi\in \bZ$ implies that $d_\chi=0$ for all $\chi\in \frakX$.
\begin{lemma}
\label{LM:characters}
Let $\frakX\subset \bfchi(H)$ be a finite set. Then $\frakX$ is a basis of $\bfchi(H)$ if and only if $\frakX$ is multiplicatively independent and $\cap_{\chi\in \frakX}\ker(\chi)$ is generated by unipotent elements.
\end{lemma}
\begin{proof}
Since $\frakX$ is a basis of $\bfchi(H)$ as a free abelian group, $\frakX$ is multiplicatively independent and $\cap_{\chi\in \frakX}\ker(\chi)=\cap_{\chi\in \bfchi(H)}\ker(\chi)$. By Lemma B.10 of \cite{feng1}, $\cap_{\chi\in \bfchi(H)}\ker(\chi)$, which is denoted by $H^t$ in \cite{feng1}, is generated by unipotent elements. This proves the necessary part. For the sufficient part, it suffices to show that $\frakX$ generates $\bfchi(H)$. Denote $\bar{H}=H/\cap_{\chi\in \frakX}\ker(\chi)$. For each $\chi\in \bfchi(H)$, by Lemma B.10 of \cite{feng1}, any unipotent element of $H$ is contained in $\ker(\chi)$ and thus from the assumption $\cap_{\chi'\in\frakX} \ker(\chi')\subset \ker(\chi)$. This implies that $\cap_{\chi\in \bfchi(H)}\ker(\chi)=\cap_{\chi\in\frakX} \ker(\chi)$. Then one has that $\bfchi(H)\cong \bfchi(\bar{H})$ (see Exercise 12 on page 108 of \cite{humphreys}). Here the isomorphism sends $\chi$ to $\bar{\chi}$, where $\bar{\chi}: \bar{H}\rightarrow \bG_m(\overline{k(\bX)})$ is given by $\bar{\chi}(\bar{\bfc})=\chi(\bfc)$ for all $\bfc\in H$.
Since $\frakX$ is multiplicatively independent, so is $\{\bar{\chi} | \chi\in \frakX\}$. Thus $\{\bar{\chi} | \chi\in \frakX\}$ is a basis of $\bfchi(\bar{H})$ and then $\frakX$ is a basis of $\bfchi(H)$.
\end{proof}

\begin{proposition}
\label{PROP:character}
There is an open subset $U$ of $\bX$ satisfying that for any $\bfc\in U$, $G_\bfc$ is an algebraic group of $\dim(G)$ and $v_\bfc(\frakX)$ is a basis of $\bfchi(G_\bfc^\circ)$.
\end{proposition}
\begin{proof}
Let $p_\chi, \chi\in\frakX$ be distinct primes. By Lemma C on page 104 of \cite{humphreys}, there is $g\in G$ such that $\chi(g)=p_\chi$ for all $\chi\in\frakX$. By Lemma~\ref{LM:kolchinlemma} and an argument similar to that in the proof of Lemma~\ref{LM:unipotents}, we may assume that the entries of $g$ are in $k[\bX]$.  Applying Proposition~\ref{PROP:groupsfibre} to $G$ and $G^\circ$ respectively, one gets a nonempty open subset $U_1$ of $\bX$ such that for any $\bfc\in U_1$, $G_\bfc$ is an algebraic group of dimension $\dim(G)$ and $(G^\circ)_\bfc$ is a connected algebraic group of dimension $\dim(G^\circ)$. By an argument similar to that in Remark~\ref{RM:corollaries}, we may assume that $(G^\circ)_\bfc\subset G_\bfc$ for all $\bfc\in U_1$. Then the dimension argument implies that $(G^\circ)_\bfc=G_\bfc^\circ$ for all $\bfc\in U_1$. We shall prove that $v_\bfc(\frakX)$ is a basis of $\bfchi((G^\circ)_\bfc)$ for all $\bfc$ being in some nonempty open subset of $\bX$. By Remark~\ref{RM:corollaries}, there is a nonempty open subset $U_2$ of $\bX$ such that for any $\bfc\in U_2$, $v_\bfc(\frakX)\subset \bfchi((G^\circ)_\bfc)$. Set $H=\cap_{\chi\in\frakX} \ker(\chi)$. Then $H$ is defined over $k[\bX]$ and by Lemma B.10 of \cite{feng1}, $H$ is generated by unipotent elements. Let $U_3$ be a nonempty open subset of $\bX$ such that $H_\bfc$ is an algebraic group generated by unipotent elements and $H_\bfc=\cap_{\chi\in \frakX}\ker(v_\bfc(\chi))$. Such $U_3$ exists due to Lemma~\ref{LM:unipotents} and Remark~\ref{RM:corollaries}. Now set $U=U_1\cap U_2\cap U_3\cap \bX_f$ with $f=\det(g)$. Assume that $\bfc\in U$. We claim that $v_\bfc(\frakX)$ is a basis of $\bfchi((G^\circ)_\bfc)$. Since the $p_\chi$ are distinct primes, for any integers $\mu_\chi, \chi\in \frakX$, not all zero,
$$
   \prod_{\chi\in \frakX} v_\bfc(\chi)(g(\bfc))^{\mu_\chi}=v_\bfc\left(\prod_{\chi\in\frakX} \chi(g)^{\mu_\chi}\right)=\prod_{\chi\in \frakX} p_\chi^{\mu_\chi}\neq 1.
$$
This implies that $v_\bfc(\frakX)$ is multiplicatively independent. Due to Lemma~\ref{LM:characters} and the fact that $\cap_{\chi\in \frakX}\ker(v_\bfc(\chi))$ is generated by unipotent elements, $v_\bfc(\frakX)$ is a basis of $\bfchi((G^\circ)_\bfc)$ and thus a basis of $\bfchi(G_\bfc^\circ)$.
\end{proof}

\section{Difference equations under specialization}
\label{SEC:differencequations}
Let $B\in \GL_n(k(\bX)(x))$ and $\sigma$ be the $k(\bX)$-automorphism of $k(\bX)(x)$ which sends $x$ to $x+1$. By setting $\sigma(X)=BX$, the automorphism $\sigma$ can be extended to an automorphism of $k(\bX)(x)[X, 1/\det(X)]$. As we shall deal with a family of automorphisms, to void confusion, the automorphism of $k(\bX)(x)[X, 1/\det(X)]$ induced by $\sigma(X)=BX$ will be denoted by $\sigma_B$. An ideal $I$ of $k(\bX)(x)[X,1/\det(X)]$ is called a $\sigma_B$-ideal if $\sigma_B(I)=I$.
Let $A$ be given as in (\ref{EQ:differenceeqn}). For convenience, we introduce the following notation.
\begin{notation}
\label{NT:well-defined}
Denote by $\bX_\frakh$ the set of $\bfc\in \bX$ satisfying that $A(\bfc)$ is well-defined and invertible. One easily see that $\bX_{\frakh}$ is open and nonempty.
\end{notation}

\begin{definition}
\label{DEF:nu-maximalspace}
Let $\nu$ be a positive integer and $I\subset \overline{k(\bX)}(x)[X, 1/\det(X)]$ a $\sigma_A$-ideal generated by some polynomials in $\overline{k(\bX)}(x)[X]_{\leq \nu}$. $I$ is said to be a $\nu$-maximal $\sigma_A$-ideal if it is not the whole ring and for any $\sigma_A$-ideal $J$ generated by some polynomials in $\overline{k(\bX)}(x)[X]_{\leq \nu}$ if $I\subset J$ then either $I=J$ or $J$ is the whole ring. Likewise we define $\nu$-maximal $\sigma_{A(\bfc)}$-ideals in $k(x)[X,1/\det(X)]$.
\end{definition}
Let $I_\nu$ be a $\nu$-maximal $\sigma_A$-ideal and let $\F$ be a fundamental matrix of ${\color{blue}\sigma_A}(Y)=AY$ over $\overline{k(\bX)}(x)$ satisfying that it is a zero of $I_\nu$. Then one has that
$$
    \left\langle\left\{p\in \overline{k(\bX)}(x)[X]_{\leq \nu} \,\,| \,\,p\left(\F\right)=0\right\}\right\rangle_{\overline{k(\bX)}(x)} \subset I_\nu
$$
where $\langle * \rangle_{\overline{k(\bX)}(x)}$ denotes the ideal in $\overline{k(\bX)}(x)[X,1/\det(X)]$ generated by $*$.
Proposition 3.5 of \cite{feng2} implies that the above two sets coincide. From this, one sees that if $J$ is another $\nu$-maximal $\sigma_A$-ideal, then there is $g\in \GL_n(\overline{k(\bX)})$ such that
$$
   J=\left\{\,\,p(Xg)\,\,|\,\, p\in I_\nu\,\,\right\}.
$$
Let $m$ be a nonnegative integer. Set
\begin{equation}
\label{EQ:partial}
\rI(m,I_\nu)=I_\nu\cap \overline{k(\bX)}[x]_{\leq m}[X]_{\leq \nu}
\end{equation}
where
$$
   \overline{k(\bX)}[x]_{\leq m}[X]_{\leq \nu}=\{p\in \overline{k(\bX)}[x,X] \,\,|\,\, \deg_x(p)\leq m, \deg_X(p)\leq \nu\}.
$$
As $I_\nu$ is finitely generated, there is an integer $\mu$ such that $\rI(\mu,I_\nu)$ generates $I_\nu$ as an ideal in $\overline{k(\bX)}(x)[X,1/\det(X)]$. We call such $\mu$ a coefficient bound of $I_\nu$. The discussion above implies that if $\mu$ is a coefficient bound of $I_\nu$ then it is a coefficient bound of any $\nu$-maximal $\sigma_A$-ideals. Hence the following definition is reasonable.
\begin{definition}
An integer $\mu$ is called a coefficient bound of $\nu$-maximal $\sigma_A$-ideals if for every $\nu$-maximal $\sigma_A$-ideal $I_\nu$, $\rI(\mu,I_\nu)$ generates $I_\nu$ as an ideal in $\overline{k(\bX)}(x)[X,1/\det(X)]$.
\end{definition}
\begin{remark}
Note that in \cite{feng2} we use the symbol $I_{\F,\nu}$ to denote the $\nu$-maximal $\sigma_A$-ideal $I_\nu$, where $\F$ is a fundamental matrix of (\ref{EQ:differenceeqn}).
\end{remark}

Let us sketch the main results of this section. First, we show that there is a coefficient bound of $I_\nu$, say $\mu$, satisfying that it is a coefficient bound of $\nu$-maximal $\sigma_{A(\bfc)}$-ideals for all $\bfc$ in some basic open subset of $\bX$ (see Lemma~\ref{LM:coefficientbound}). Second, under the hypothesis that $\rI(\mu, I_\nu)$ has an $\overline{k(\bX)}$-basis $B$ contained in $k[\bX][x,X]$, we prove that there is a basic open subset of $\bX$ such that for each $\bfc$ in this set, $v_\bfc(B)$ is a basis of $\rI(\mu, \tilde{I}_\nu)$ as a $k$-vector space for some $\nu$-maximal $\sigma_{A(\bfc)}$-ideal $\tilde{I}_\nu$. The choice of $\mu$ implies that $v_\bfc(B)$ generates $\tilde{I}_\nu$ (see Proposition~\ref{PROP:maximalideals}).

Before we go further, let us first introduce some notations. Let $X^{\bfd_1},\cdots,X^{\bfd_\ell}$ be all monomials in $X$ with degree not greater than $\nu$, where $\ell=\binom{n^2+\nu-1}{\nu}$. Then $\{X^{\bfd_1},\cdots,X^{\bfd_\ell}\}$ is a basis of $\overline{k(\bX)}(x)[X]_{\leq \nu}$ as a vector space over $\overline{k(\bX)}(x)$. Let $Y$ be an $n\times n$ matrix with indeterminate entries and for a matrix $M$, let $M^t$ denote its transpose.
\begin{notation}
\label{NT:sym}
Suppose that $F$ is an $n\times n$ matrix with entries in a $\overline{k(\bX)}(x)$-algebra $R$. Then the map  sending $X$ to $FX$ induces a map
\begin{align*}
   \Sym_\nu: \Mat_n(R) &\longrightarrow \Mat_\ell(R) \\
                F & \longrightarrow \Sym_\nu(F)
\end{align*}
where $\Sym_\nu(F)$ is defined to be the matrix satisfying that
$$
   (X^{\bfd_1},\cdots,X^{\bfd_\ell})^t|_{X=FY}=\Sym_\nu(F)(X^{\bfd_1},\cdots,X^{\bfd_\ell})^t|_{X=Y}.
$$
\end{notation}
Let $l$ be a positive integer not greater than $n$. Denote by $\I_{n,l}$ the set of all subsets of $\{1,2,\cdots,n\}$ containing exactly $l$ elements. We define an order $\prec$ on $\I_{n,l}$ as follows: for $\bfi,\bfj\in \I_{n,l}$, $\bfi\prec\bfj$ if they satisfy that (1) $\min \bfi <\min \bfj$
or (2) $\min \bfi =\min \bfj$ and $\bfi \setminus \{\min \bfi\}\prec \bfj \setminus \{\min \bfj\}$.
\begin{notation}
\label{NT:minor} We use $\Phi_{n,l}$ to denote the map defined as follows:
\begin{align*}
    \GL_n(\overline{k(\bX)}(x))&\longrightarrow \GL_{\binom{n}{l}}(\overline{k(\bX)}(x)) \\
        Z  &\longrightarrow \left(Z_{\bfi,\bfj}\right)_{\{1,2,\cdots,l\}\prec \bfi,\bfj \prec \{n-l+1,\cdots,n\}}
\end{align*}
where $Z_{\bfi,\bfj}$ denotes the $l\times l$ minor of $Z$ that corresponds to the rows with index in $\bfi$ and the columns with index in $\bfj$.
\end{notation}
\begin{remark}
\label{RM:NT}
\begin{itemize}
\item [$(1)$] By the definition, one sees that
  $$\Sym_\nu(F_1F_2)=\Sym_\nu(F_1)\Sym_\nu(F_2)$$
 and if $F$ is invertible then so is $\Sym_\nu(F)$.
\item [$(2)$] Write $F=(f_{i,j})$ with $f_{i,j}\in R$. Then the vector space spanned by the entries of $\Sym_\nu(F)$ coincides with the one spanned by $\prod_{i,j}f_{i,j}^{s_{i,j}}$ with $0\leq \sum_{i,j}s_{i,j}\leq \nu$. To see this, let $V$ denote the latter vector space. Obviously, all entries of $\Sym_\nu(F)$ are in $V$. On the other hand, by the definition of $\Sym_\nu$, one has that
$$
   (\cdots, \prod_{i,j}f_{i,j}^{s_{i,j}},\cdots)^t=(X^{\bfd_1},\cdots,X^{\bfd_\ell})^t|_{X=F}=\Sym_\nu(F)(X^{\bfd_1},\cdots,X^{\bfd_\ell})^t|_{X=I_n},
$$
which implies that each $\prod_{i,j}f_{i,j}^{s_{i,j}}$ is a $\CQ$-combination of the entries of $\Sym_\nu(F)$. Hence these vector spaces are equal to each other.
\item [$(3)$] One sees that $\Phi_{n,l}(I_n)=I_{\binom{n}{l}}$ and if $M$ is a permutation matrix then so is $\Phi_{n,l}(M)$. Furthermore, the Cauchy-Binet formula (see Proposition 2.1.2 on page 18 of \cite{dserre}) implies that $\Phi_{n,l}$ is actually a group homomorphism.
\end{itemize}
\end{remark}
\subsection{Coefficient bounds of $\nu$-maximal $\sigma_A$-ideals}
\label{SUBSEC:coefficientbounds}
In this subsection, we shall show that there is a coefficient bound $N$ of $\nu$-maximal $\sigma_A$-ideals and a basic open subset $U$ of $\bX$ such that $N$ is also a coefficient bound of $\nu$-maximal $\sigma_{A(\bfc)}$-ideals for all $\bfc\in U$. Such a coefficient bound can be derived from a degree bound of the certificates of hypergeometric solutions of a suitable linear difference equation.
\begin{definition}
\label{DEF:hypergeometric}
Let $R$ be a $\sigma$-extension ring of $\overline{k(\bX)}(x)$. $h\in R$ is said to be hypergeometric over $\overline{k(\bX)}(x)$ if $h$ is invertible in $R$ and $\sigma(h)h^{-1}\in \overline{k(\bX)}(x)$, which is called the certificate of $h$. A solution $\bfh$ of (\ref{EQ:differenceeqn}) is called a hypergeometric solution if $\bfh=\bfv h$ where $\bfv\in \overline{k(\bX)}(x)^n$ and $h$ is hypergeometric over $\overline{k(\bX)}(x)$.
\end{definition}

Let us recall the method developed in \cite{feng2} to compute a coefficient bound of a $\nu$-maximal $\sigma_A$-ideal $I_\nu$. Denote
$$
    S_\nu=I_\nu\cap \overline{k(\bX)}(x)[X]_{\leq \nu}.
$$
Then $S_\nu$ is a $\overline{k(\bX)}(x)$-vector space of finite dimension and it generates $I_\nu$. Suppose that $\{p_1,\cdots,p_l\}$ is a $\overline{k(\bX)}(x)$-basis of $S_\nu$. Let $X^{\bfd_1},\cdots,X^{\bfd_\ell}$ be as in Notation~\ref{NT:sym}. After an invertible linear transformation of $p_1,\cdots,p_l$ if necessary, we may assume that for each $i=1,\cdots,l$
\begin{equation}
\label{EQ:coefficients}
   p_i=X^{\bfd_i}+\sum_{j=l+1}^\ell c_{i,j}X^{\bfd_j}
\end{equation}
with $c_{i,j}\in \overline{k(\bX)}(x)$. For $f\in \overline{k(\bX)}(x)\setminus \{0\}$, $\deg(f)$ stands for the degree of $f$ which is defined to be the maximum of the degrees of its numerator and denominator. For convenience, set $\deg(0)=-\infty$. Then we have following claim.
\begin{claim}
\label{claim}
 $\ell m$ is a coefficient bound of $I_\nu$ if $m$ is not less than $\deg(c_{i,j})$ for all $i,j$.
\end{claim}
Clearing the denominators of $c_{i,j}$ in (\ref{EQ:coefficients}), we obtain $\tilde{p}_i\in \overline{k(\bX)}[x,X]$ with $\deg_X(\tilde{p}_i)\leq \nu$ and $\deg_x(\tilde{p}_i)\leq (\ell-l)m< \ell m$. In other words, $\tilde{p}_i\in \rI(\ell m,I_\nu)$ and $\{\tilde{p}_1,\cdots,\tilde{p}_l\}$ is a basis of $S_\nu$, where $\rI(\ell m,I_\nu)$ is defined as in (\ref{EQ:partial}).
Hence $\ell m$ is a coefficient bound of $I_\nu$. This proves our claim. So in order to obtain a coefficient bound of $I_\nu$, it suffices to compute a degree bound of $c_{i,j}$. In the following, we show that a degree bound of $c_{i,j}$ can be achieved via computing the certificates of hypergeometric solutions of certain linear difference equations.  We have that
$$
   \sigma_A ((X^{\bfd_1},\cdots,X^{\bfd_\ell})^t)=\Sym_\nu(A)(X^{\bfd_1},\cdots,X^{\bfd_\ell})^t
$$
where $\Sym_\nu$ is defined as in Notation~\ref{NT:sym}.
From (\ref{EQ:coefficients}), $\{p_1,\cdots,p_l, X^{\bfd_{l+1}},\cdots,X^{\bfd_\ell}\}$ is another $\overline{k(\bX)}(x)$-basis of $\overline{k(\bX)}(x)[X]_{\leq \nu}$, and moreover one has that
\begin{equation}
\label{EQ:transformation}
   (p_1,\cdots,p_l,X^{\bfd_{l+1},\cdots,X^{\bfd_\ell}})^t=\begin{pmatrix}
         I_l & C\\
         0 & I_{\ell-l}
   \end{pmatrix}(X^{\bfd_1},\cdots,X^{\bfd_\ell})^t
\end{equation}
where $C=(c_{i,j})_{1\leq i \leq l, l+1\leq j\leq \ell}$ with $c_{i,j}$ given in (\ref{EQ:coefficients}). Since $S_\nu$ is stable under the action of $\sigma_A$, one has that
$$
  \sigma_A((p_1,\cdots,p_l,X^{\bfd_{l+1}},\cdots,X^{\bfd_\ell})^t)=\begin{pmatrix}
         B_1 & 0\\
         B_2 & B_3\end{pmatrix}(p_1,\cdots,p_l,X^{\bfd_{l+1}},\cdots,X^{\bfd_\ell})^t
$$
where $B_1\in \GL_l(\overline{k(\bX)}(x)), B_3\in \GL_{\ell-l}(\overline{k(\bX)}(x))$ and $B_2$ is an $(\ell-l)\times l$ matrix with entries in $\overline{k(\bX)}(x)$.
Applying $\sigma_A$ to (\ref{EQ:transformation}) yields that
\begin{equation*}
 \begin{pmatrix}
         I_l &\sigma(C) \\
          0& I_{\ell-l}
   \end{pmatrix}\Sym_\nu(A)(X^{\bfd_1},\cdots,X^{\bfd_\ell})^t=\begin{pmatrix}
         B_1 & 0\\
         B_2 & B_3
   \end{pmatrix}\begin{pmatrix}
         I_l & C\\
         0 & I_{\ell-l}
   \end{pmatrix}(X^{\bfd_1},\cdots,X^{\bfd_\ell})^t.
\end{equation*}
As $X^{\bfd_1},\cdots,X^{\bfd_\ell}$ are linearly independent over $\overline{k(\bX)}(x)$, the above equality implies
\begin{equation}
\label{EQ:equality}
 \begin{pmatrix}
         I_l &\sigma(C) \\
          0& I_{\ell-l}
   \end{pmatrix}\Sym_\nu(A)=\begin{pmatrix}
         B_1 & 0\\
         B_2 & B_3
   \end{pmatrix}\begin{pmatrix}
         I_l & C\\
         0 & I_{\ell-l}
   \end{pmatrix}.
\end{equation}
Denote by $\bfs$ the first row of
$$
    \Phi_{\ell,l}\left(\begin{pmatrix}
         I_l & C\\
         0 & I_{\ell-l}
   \end{pmatrix}\right)
$$
where $\Phi_{\ell,l}$ is defined as in Notation~\ref{NT:minor}.
Applying $\Phi_{\ell,l}$ to (\ref{EQ:equality}), we obtain that
$$
    \begin{pmatrix}\sigma(\bfs)\\ * \end{pmatrix}\Phi_{\ell,l}\left(\Sym_\nu(A)\right)=\begin{pmatrix}
         \det(B_1) & 0\\
          * & *
   \end{pmatrix}\begin{pmatrix}
   \bfs \\ 0
   \end{pmatrix}
$$
which implies that
\begin{equation*}
\sigma(\bfs)\Phi_{\ell,l}\left(\Sym_\nu(A)\right)=\det(B_1)\bfs.
\end{equation*}
Let $h$ be the hypergeometric element in some $\sigma$-extension ring of $\overline{k(\bX)}(x)$ with $\det(B_1)$ as its certificate. Then $\bfs^t h$ is a hypergeometric solution of the following linear difference equation
\begin{equation}
\label{EQ:differenceeqn2}
 \sigma_{\Phi_{\ell,l}\left(\Sym_\nu(A)\right)^{-t}}(Y)=\Phi_{\ell,l}\left(\Sym_\nu(A)\right)^{-t}Y,
\end{equation}
where $*^{-t}$ denotes the transpose of the inverse of $*$. For each $\bfi\in \I_{\ell,l}$, denote by $s_{\bfi}$ the $l\times l$-minor of $(I_l, C)$ corresponding to the columns with index in $\bfi$. Then $\bfs=(s_\bfi)_{\bfi\in \I_{\ell,l}}$ and one can verify that
\begin{equation}
\label{EQ:coefficients2}
   s_\bfi=
   \begin{cases}
        1, &\bfi=\{1,2,\cdots,l\} \\
        (-1)^{l-j} c_{i,j}, & \bfi=\{1,2,\cdots,i-1,i+1,\cdots,l,j\}
 \end{cases}
\end{equation}
for all $i\in \{1,2,\cdots,l\}$ and all $j\in \{l+1,\cdots,\ell\}$. Therefore to compute a degree bound for $c_{i,j}$, we only need to compute a degree bound for entries of $\bfs$.

It is well-known that the equation (\ref{EQ:differenceeqn2}) is equivalent to a linear difference operator with coefficients in $\overline{k(\bX)}(x)$ (see Section 1 of \cite{birkhoff}). Precisely, there is a matrix $T\in \GL_\mu(\overline{k(\bX)}(x))$ such that $\sigma(T)\Phi_{\ell,l}(\Sym_\nu(A))^{-t}T^{-1}$ is of the form
$$
   \begin{pmatrix}
      0 & 1 & 0 & \cdots & 0   \\
      \vdots  & \ddots & \ddots & \ddots & \vdots   \\
        & &  & 1 &  0   \\
      0 & \cdots & \cdots & 0 & 1  \\
      -a_0& -a_1 & \cdots &\cdots & -a_{\mu-1}
   \end{pmatrix}
$$
where $\mu=\binom{\ell}{l}=|\I_{\ell,l}|$, the order of the matrix $\Phi_{\ell,l}(\Sym_\nu(A))$. In other words, under the transformation $T$, the equation (\ref{EQ:differenceeqn2}) is equivalent to
$$
   \L=\sigma^\mu+a_{\mu-1}\sigma^{\mu-1}+\cdots+a_0,
$$
and the solution $\bfs^t h$ of (\ref{EQ:differenceeqn2}) is transformed into
$$
   T\bfs^t h =\begin{pmatrix}
         1 \\
         \tilde{r}\\
         \vdots\\
         \prod_{i=0}^{\mu-2}\sigma^i(\tilde{r})
   \end{pmatrix}\tilde{h}.
$$
where $\tilde{h}$ is a hypergeometric solution of $\L(y)=0$ and $\tilde{r}$ is the certificate of $\tilde{h}$.
Denote by $\deg(T^{-1})$ the maximum of the degrees of entries of $T^{-1}$ and denote
$$
   \begin{pmatrix}
         w_{\bfi_1} \\
         w_{\bfi_2}\\
         \vdots\\
         w_{\bfi_\mu}
   \end{pmatrix}=T^{-1}\begin{pmatrix}
         1 \\
         \tilde{r}\\
         \vdots\\
         \prod_{i=0}^{\mu-2}\sigma^i(\tilde{r})
   \end{pmatrix}.
$$
Since $\deg(\prod_{i=0}^j\sigma^i(\tilde{r}))\leq (j+1)\deg(\tilde{r})$ for all $0\leq j \leq \mu-2$,
$$\deg(w_{\bfi_j})\leq \mu \deg(T^{-1})+\mu(\mu-1)\deg(\tilde{r})$$
for all $j=1,\cdots,\mu$. On the other hand, since $\bfs=(w_{\bfi_1},\cdots,w_{\bfi_\mu})h^{-1} \tilde{h}$, by (\ref{EQ:coefficients2}),
\begin{equation}
\label{EQ:inequality}
 \deg(c_{i,j})=\deg(s_{\bfj}/s_{\bfi_1})=\deg(w_{\bfj}/w_{\bfi_1})\leq 2 \mu \deg(T^{-1})+2\mu(\mu-1)\deg(\tilde{r})
\end{equation}
where $\bfi_1=\{1,2,\cdots,l\}, \bfj= \{1,\cdots,i-1,i+1,\cdots,l,j\}$. Therefore to bound the degree of $c_{i,j}$, it suffices to bound the degrees of the certificates of all hypergeometric solutions of $\L(y)=0$. For the latter purpose, we introduce the following definition.
\begin{definition}
\label{DEF:hyper-bound}
 A nonnegative integer $N$ is call a hyper-bound for $\L$ if the certificates of all hypergeometric solutions of $\L(y)=0$ are of degree $\leq N$.
\end{definition}

\begin{remark}
\begin{itemize}
\item [(1)]
In the above discussion, we need to priorly know how large the dimension of $S_\nu$ is. In the case when this dimension can not be determined priorly, we can compute hyper-bounds for linear difference operators corresponding to $\sigma_{\Phi_{\ell,l}(\Sym_\nu(A))^{-t}}(Y)=\Phi_{\ell,l}(\Sym_\nu(A))^{-t}Y$ with $l=1,2,\cdots,\ell$. Each hyper-bound gives a potential coefficient bound of $I_\nu$. The maximum of these potential coefficient bounds will be what we need.
\item [(2)] The method described above also works for linear difference equations with coefficients in $k(x)$. Particularly, let $\bfc\in \bX_\frakh$ where $\bX_\frakh$ is given in Notation~\ref{NT:well-defined}. We can find a coefficient bound for $\nu$-maximal $\sigma_{A(\bfc)}$-ideals from hyper-bounds for linear difference operators corresponding to $\sigma_{\Phi_{\ell,l}(\Sym_\nu(A(\bfc)))^{-t}}(Y)=\Phi_{\ell,l}(\Sym_\nu(A(\bfc)))^{-t}Y$ with $l=1,2\cdots,\ell$.
\end{itemize}
\end{remark}

In the rest of this subsection, we shall deal with hyper-bounds for a linear difference operator $\L$. After multiplying a polynomial in $\overline{k(\bX)}[x]$, we may assume that $\L$ has polynomial coefficients, i.e.
$$
     \L=a_n(x)\sigma^n+\cdots+a_1(x)\sigma+a_0(x)
$$
with $a_i(x)\in \overline{k(\bX)}[x]$ and $a_n(x)a_0(x)\neq 0$. Let us first investigate polynomial solutions. Set $\bar{\sigma}=x(\sigma-\id)$. Multiplying $\L$ with a suitable polynomial in $\bZ[x]$, one obtains a new operator of the form $\sum_{i=0}^n \bar{a}_i(x)\bar{\sigma}^i\in \overline{k(\bX)}[x][\bar{\sigma}]$. Denote
$$
    \rho=\max \{\deg(\bar{a}_0), \cdots, \deg(\bar{a}_n)\}.
$$
\begin{definition}
$\sum_{i=0}^n \coeff(\bar{a}_i,x,\rho)y^i$ is called the indicial polynomial of $\L$, denoted by $\Ind(\L)$, where $\coeff(\bar{a}_i,x,\rho)$ denotes the coefficient of $x^\rho$ in $\bar{a}_i$.
\end{definition}
\begin{remark}
\label{RM:polysols} Let $p(x)=c x^m+c_{m-1}x^{m-1}+\cdots+c_0$ be a polynomial of degree $m$. Then for each $i=0,\cdots,n$, one has that
$$
    \bar{\sigma}^i(p(x))=c m^i x^m +\mbox{terms of lower degree}.
$$
Furthermore,
$$
   \L(p(x))=c\left(\sum_{i=0}^n\coeff(\bar{a}_i,x,\rho)m^i\right)x^{\rho+m}+\mbox{terms of lower degree}.
$$
Therefore if $\L(p(x))=0$ then $m$ is an integer zero of $\Ind(\L)$.
\end{remark}
Assume that $\L\in k[\bX][x, \sigma]$. For $\bfc\in \bX$, $\L_\bfc$ denotes the operator obtained by applying $v_\bfc$ to the coefficients of $\L$.
\begin{lemma}
\label{LM:polysols}
Let $N=\max \,\, \Z(\Ind(\L))\cup \{0\}$. Then there is a basic open subset $U$ of $\bX$ such that polynomial solutions of $\L_\bfc(y)=0$ with $\bfc\in U$ are of degree not greater than $N$.
\end{lemma}
\begin{proof}
By Lemma~\ref{LM:integers}, there is a finitely generated subgroup $\Gamma$ of $\bG_a(\overline{k(\bX)})$ such that for any $\bfc\in \B(\bX,\Gamma)$,  $\Z(\Ind(\L))=\Z(v_\bfc(\Ind(\L)))$. Let $c$ be a nonzero element in $k[\bX]$ such that for every $\bfc\in \bX_c$, $\deg(\bar{a}_i)=\deg(v_\bfc(\bar{a}_i))$ for all $0\leq i \leq n$. Let $U=\B(\bX,\Gamma)\cap \bX_c$. Suppose that $\bfc\in U$. One has that $v_\bfc(\Ind(\L))=\Ind(\L_\bfc)$ and then
\begin{align*}
    \max \,\,\Z(\Ind(\L_\bfc))\cup \{0\} &=\max \,\,\Z(v_\bfc(\Ind(\L)))\cup \{0\}\\
   & =\max \,\,\Z(\Ind(\L))\cup \{0\}=N.
\end{align*}
By Remark~\ref{RM:polysols}, every polynomial solution of $\L_\bfc$ has degree not greater than $ N$.
\end{proof}
To investigate the behavior of the certificates of hypergeometric solutions under specialization, we need to recall the algorithm given in \cite{petkovsek} for finding hypergeometric solutions.
Denote
$$
   \calS_\L=\left\{\left.(p,q)\in \overline{k(\bX)}[x] \,\,\right| \,\,\mbox{$p,q$ are monic and $p|a_0(x), q|a_n(x-n+1)$}\right\}.
$$
\begin{algorithm}
\label{ALG:hyper}
Input: $\L(y):=\sum_{i=0}^n a_i(x)\sigma^i(y)$ with polynomial $a_i(x)$\\
Output: the certificate of a hypergeometric solution of $\L(y)=0$ if there exists; otherwise 0.
\begin{itemize}
\item [$(a)$]
For each $(p,q)\in \calS_\L$ do
\begin{itemize}
\item [$(1)$]
   $P_i(x):=a_i(x)\prod_{j=0}^{i-1}p(x+j)\prod_{j=i}^{n-1}q(x+j)$ for all $i=0,1,\cdots,n$;
\item [$(2)$]
   $m: =\max\{\deg(P_i(x))\}$ and  $\alpha_i:=\coeff(P_i(x),x, m)$ for all $0\leq i\leq n$;
\item [$(3)$] let $\calZ_{p,q}\subset \overline{k(\bX)}$ be the set of all nonzero solutions of
   $$f_{p,q}(y)=\sum_{i=0}^n \alpha_i y^i=0;$$
\item [$(4)$] for each $\beta\in \calZ_{p,q}$ do if the linear difference equation
     $$
        \L_{p,q,\beta}=\sum_{i=0}^n \beta^i P_i(x)\sigma^i =0
     $$
     has a nonzero polynomial solution $Q(x)$, then return
     $$\beta\frac{p(x)}{q(x)}\frac{Q(x+1)}{Q(x)}. $$
     Note that one can test if $\L_{p,q,\beta}(y)=0$ has a polynomial solution by Algorithm Poly in \cite{petkovsek}.
\end{itemize}
\item [$(b)$] Return 0.
\end{itemize}
\end{algorithm}
 Let $\calS_\L, \calZ_{p,q}, \L_{p,q,\beta}$ be as in Algorithm~\ref{ALG:hyper}. We set
$$
   N(\L)= \max \,\,\{0\}\cup\Z\left(\prod_{(p,q)\in \calS_\L, \beta\in \calZ_{p,q} }\Ind(\L_{p,q,\beta})\right)
         +\max\{\deg(a_n), \deg(a_0)\}.
$$
Remark that the above algorithm allows one to compute the certificates of all hypergeometric solutions of $\L$. The certificates in the output are of degree not greater than
{\small\begin{align*}
 \mathop{\max}\limits_{(p,q)\in \calS_\L,Q\in \P} \left\{\deg\left(\frac{p(x)}{q(x)}\frac{Q(x+1)}{Q(x)}\right)\right\}&\leq  \mathop{\max}\limits_{(p,q)\in \calS_\L,Q\in \P}\{\deg(p)+\deg(Q), \deg(q)+\deg(Q)\}\\
 &\leq \max\{\deg(a_0),\deg(a_n)\}+\max_{Q\in\P}\{\deg(Q)\},
\end{align*}
}where $\P$ is the set of all polynomial solutions of $\L_{p,q,\beta}(y)=0$ for all $(p,q)\in \calS_\L$ and $\beta\in \calZ_{p,q}$.
If $Q(x)$ is a polynomial solution of $\L_{p,q,\beta}(y)=0$ then $\deg(Q)$ is not greater than $\max \{0\}\cup \Z(\Ind(\L_{p,q,\beta}))$ by Remark~\ref{RM:polysols}.
Therefore by definition $N(\L)$ is a hyper-bound for $\L$. Moreover, we have the following result.
\begin{lemma}
\label{LM:hyper}
There is a basic open subset $U$ of $\bX$ such that for any $\bfc\in U$, $N(\L)$ is a hyper-bound for $\L_\bfc$.
\end{lemma}
\begin{proof}
 Let $\calS_\L, f_{p,q},\calZ_{p,q}, \L_{p,q,\beta}$ be as in Algorithm~\ref{ALG:hyper} and let
$$
    W=\{1,\lc(a_0(x)), \cdots, \lc(a_n(x))\}\bigcup\bV(a_n(x))\bigcup\bV(a_0(x))
     \bigcup\bigcup_{(p,q)\in \calS_\L} \calZ_{p,q}
$$
 where $\bV(a_i(x))$ denotes the set of roots of $a_i(x)=0$ in $\overline{k(\bX)}$.
Let $\tilde{D}\subset \overline{k(\bX)}$ be a finitely generated $k[\bX]$-algebra such that $W\subset \tilde{D}$ and $\bY$ the variety over $k$ associated to $\tilde{D}$. Let $\Gamma$ be the subgroup of $\bG_a(\overline{k(\bX)})$ generated by $W$. Suppose that $\bfc\in \B(\bY,\Gamma)$. It is easy to see that $\calS_{\L_\bfc}=v_\bfc(\calS_\L)$. Furthermore one sees that for each $(v_\bfc(p),v_\bfc(q))\in \calS_{\L_\bfc}$,
 $$
    f_{v_\bfc(p),v_\bfc(q)}=v_\bfc(f_{p,q}),\,\, \calZ_{v_\bfc(p),v_\bfc(q)}=v_\bfc(\calZ_{p,q}),
 $$
 and for each $\beta\in \calZ_{p,q}$, $\L_{v_\bfc(p),v_\bfc(q),v_\bfc(\beta)}=v_\bfc(\L_{p,q,\beta})$.
 This together with Algorithm~\ref{ALG:hyper} implies that all certificates of hypergeometric solutions of $\L_\bfc(y)=0$ are of the form
\begin{equation}
\label{EQ:certificate}
     v_\bfc(\beta)\frac{v_\bfc(p)}{v_\bfc(q)}\frac{\bar{Q}(x+1)}{\bar{Q}(x)}
\end{equation}
 where $(p,q)\in \calS_\L, \beta\in \calZ_{p,q}$ and $\bar{Q}(x)$ is a nonzero polynomial solution of the linear difference equation $\L_{v_\bfc(p),v_\bfc(q),v_\bfc(\beta)}(y)=0$. Now let $\tilde{U}_{p,q,\beta}$ be a basic open subset of $\bY$ such that for any $\bfc\in \tilde{U}_{p,q,\beta}$, nonzero polynomial solutions of $v_\bfc(\L_{p,q,\beta})(y)=0$ i.e. $\bar{Q}(x)$, are of degree not greater than
 $$
     \max \,\,\Z(\Ind(\L_{p,q,\beta}))\cup \{0\}.
 $$
 Such $\tilde{U}_{p,q,\beta}$ exists due to Lemma~\ref{LM:polysols}.
 Set
$$
   U=\B(\bY,\Gamma)\bigcap \bigcap_{(p,q)\in \calS_\L, \beta\in \calZ_{p,q}}\tilde{U}_{p,q,\beta}.
$$
Then for any $\bfc\in U$, the degrees of rational functions in (\ref{EQ:certificate}) are not greater than $N(\L)$ and so $N(\L)$ is a hyper-bound for $\L_\bfc$. The lemma then follows from Lemma~\ref{LM:property} and the fact that $\L_\bfc=\L_{p_{\bY/\bX}(\bfc)}$.
\end{proof}

\begin{lemma}
\label{LM:coefficientbound}
There are a coefficient bound $N$ of $\nu$-maximal $\sigma_A$-ideals and a basic open subset $U$ of $\bX$ such that $N$ is also a coefficient bound of $\nu$-maximal $\sigma_{A(\bfc)}$-ideals for all $\bfc\in U$.
\end{lemma}
\begin{proof}
The notations are as before.
 For each $l=1,2,\cdots,\ell$, by the method developed in Section 1 of \cite{birkhoff}, compute a matrix $T_l\in \GL_{\binom{\ell}{l}}(\overline{k(\bX)}(x))$ such that under the transformation $T_l$, $\sigma_{\Phi_{\ell,l}(\Sym_\nu(A))^{-t}}(Y)=\Phi_{\ell,l}(\Sym_\nu(A))^{-t}Y$ is equivalent to a linear difference operator $\L_l$. Let $N_l$ be a hyper-bound for $\L_l$. Set
 $$
    N=\ell \cdot\max_{1\leq l \leq \ell} \{ 2\tilde{\mu} \deg(T_l^{-1}) + 2\tilde{\mu}(\tilde{\mu}-1)N_l\},
$$
where $\tilde{\mu}=\max \{\binom{\ell}{l}| 1\leq l \leq \ell\}$.
Then by (\ref{EQ:inequality}) and Claim~\ref{claim}, $N$ is a coefficient bound of $\nu$-maximal $\sigma_A$-ideals. Let $\tilde{D}\subset \overline{k(\bX)}$ be a finitely generated $k[\bX]$-algebra such that the entries of $T_l, T_l^{-1}$ and $\Phi_{\ell,l}(\Sym_\nu(A))^{-t}$ are in the field of fractions of $\tilde{D}[x]$ for all $l=1,2,\cdots,\ell$, and let $\bY$ be the variety over $k$ associated to $\tilde{D}$. Take a nonzero $\tilde{h}\in \tilde{D}$ such that for any $\tilde{\bfc}\in \bY_{\tilde{h}}$ and all $l=1,2,\cdots,\ell$, $v_{\tilde{\bfc}}(T_l)$ and $A(\tilde{\bfc})$ are well-defined and invertible, and $\sigma_{\Phi_{\ell,l}(\Sym_\nu(A(\tilde{\bfc})))^{-t}}(Y)=\Phi_{\ell,l}(\Sym_\nu(A(\tilde{\bfc})))^{-t}Y$ is equivalent to the linear difference operator $v_{\tilde{\bfc}}(\L_l)$ under the transformation $v_{\tilde{\bfc}}(T_l)$. Due to Lemma~\ref{LM:hyper}, there is a basic open subset $U_1$ of $\bX$ such that $N_l$ is a hyper-bound for $v_{\bfc}(\L_l)$ for all $l=1,2,\cdots,\ell$ and all $\bfc\in U_1$. By Lemma~\ref{LM:kolchinlemma}, there is a nonempty open subset $U_2$ of $\bX$ such that $U_2\subset p_{\bY/\bX}(\bY_{\tilde{h}})$. Set $U=U_1\cap U_2$ and suppose that $\bfc\in U$. Let $\tilde{\bfc}$ be an element in $\bY_{\tilde{h}}\cap p_{\bY/\bX}^{-1}(\bfc)$. One sees that $\deg(v_{\tilde{\bfc}}(T_l^{-1}))\leq \deg(T_l^{-1})$, and by (\ref{EQ:inequality}) and Claim~\ref{claim} again,
$$
   \tilde{N}=\ell\cdot\max_{1\leq l \leq \ell} \{ 2\tilde{\mu} \deg(v_{\tilde{\bfc}}(T_l^{-1})) + 2 \tilde{\mu}(\tilde{\mu}-1)N_l\}
$$
is a coefficient bound of $\nu$-maximal $\sigma_{A(\tilde{\bfc})}$-ideals. The lemma then follows from  Lemma~\ref{LM:property} and the facts that $N\geq \tilde{N}$ and $A(\tilde{\bfc})=A(\bfc)$.
\end{proof}
\begin{remark}
The coefficient bound $N$ given in Lemma~\ref{LM:coefficientbound} only depends on the matrix $A$ and the given integer $\nu$.
\end{remark}

\subsection{$\nu$-Maximal $\sigma_A$-ideals under specialization}
\label{SUBSEC:numaximalideals}
Let $I_\nu$ be a $\nu$-maximal $\sigma_A$-ideal in $\overline{k(\bX)}(x)[X,1/\det(X)]$ and $\rI(m,I_\nu)$ as in (\ref{EQ:partial}). The aim of this subsection is to prove that $I_\nu$ is sent to a $\nu$-maximal $\sigma_{A(\bfc)}$-ideal by $v_\bfc$ for all $\bfc$ in some basic open subset of $\bX$. We shall first prove that for each $m\geq 0$ there exists a basic open subset $U$ of $\bX$ such that for any $\nu$-maximal $\sigma_{A(\bfc)}$-ideal $J_\bfc$ in $k[X,1/\det(X)]$ with $\bfc\in U$, the dimension of $\rI(m,J_\bfc)$ is equal to that of $\rI(m,I_\nu)$. To this end, we need the following definition.

\begin{definition}
\label{DEF:dimension}
The dimension of (\ref{EQ:differenceeqn}) is defined to be the dimension of the vector space over $\overline{k(\bX)}$ spanned by the entries of a fundamental matrix of (\ref{EQ:differenceeqn}), denoted by $\dim([A])$.
\end{definition}
Given a fundamental matrix $\F$ of (\ref{EQ:differenceeqn}), there is a linear difference operator $\L\in \overline{k(\bX)}(x)[\sigma]$ whose solution space is spanned by the entries of $\F$. Moreover, for such $\L$ one has that $\ord(\L)=\dim([A])$. Such $\L$ can be constructed as follows. Let $\bfv_j$ be the $j$-th column of $\F$. Then $\L$ is an operator of minimal order that annihilates all $\bfv_j$, i.e. all entries of $\bfv_j$ for all $j$. Note that as $\F$ has $n^2$ entries, by definition, $\dim([A])\leq n^2$ and thus $\ord(\L)\leq n^2$. For each $l=1,\cdots,n^2$, $\sigma^l(\bfv_j)=A_l \bfv_j$ for all $j=1,\cdots,n$, where $A_l=\sigma^{l-1}(A)\cdots \sigma(A)A$. Assume that $a_0,\cdots,a_s\in \overline{k(\bX)}(x)$ with $s\leq n^2$. Then $\sum_{l=0}^s a_l \sigma^l(\bfv_j)=0$ for all $j=1,\cdots,n$ if and only if $a_0\bfv_j+\sum_{l=1}^s a_l A_l\bfv_j=0$ for all $j=1,\cdots,n$. The later equalities are equivalent to
$$a_0I_n+\sum_{l=1}^s a_l A_l=0$$
because $\F=(\bfv_1,\cdots,\bfv_n)$ is invertible. Let $\{\bfe_1,\cdots,\bfe_n\}$ be the standard basis of $\overline{k(\bX)}(x)^n$. Set
\begin{equation}
\label{EQ:matrix}
   \M_A=\begin{pmatrix} \bfe_1^t & \bfe_2^t & \cdots & \bfe_n^t \\
                      A_1^{[1]} & A_1^{[2]} & \cdots & A_1^{[n]} \\
                      A_2^{[1]} & A_2^{[2]} & \cdots & A_2^{[n]} \\
                        \vdots &\vdots& & \vdots \\
                        A_{n^2}^{[1]} & A_{n^2}^{[2]} & \cdots & A_{n^2}^{[n]} \\
      \end{pmatrix}
\end{equation}
where $A_i^{[j]}$ denotes the $j$-th row of $A_i$.
Then $\M_A$ is a $(1+n^2)\times n^2$ matrix with entries in $\overline{k(\bX)}(x)$, and one sees that $a_0I_n+\sum_{l=1}^s a_l A_l=0$ if and only if $(a_0,\cdots,a_s,0,\cdots,0)$ is in the left kernel of $\M_A$.
Let $(b_0,\cdots,b_s, 0, \cdots, 0)$ be an element of the left kernel of $\M_A$ satisfying that $b_s\neq 0$ and $s$ is as small as possible. Then $\L$ can be chosen to be $\sum_{i=0}^s b_i \sigma^i$ and $s=\ord(\L)=\dim([A])$. The above construction indicates the following lemma.
\begin{lemma}
\label{LM:dim}
There is a nonempty open subset $U$ of $\bX$ such that if $\bfc\in U$ then $$\dim([A])=\dim([A(\bfc)]).$$
\end{lemma}
\begin{proof}
We first show that $\dim([A])=\rank(\M_A)$, where $\M_A$ is given as in (\ref{EQ:matrix}). Denote $r=\dim([A])$. If $\rank(\M_A)<r$.
then the first $r$ rows of $\M_A$ are linearly dependent over $k(\bX)(x)$. This implies that the left kernel of $\M_A$ contains a nonzero element of the form $(b_0,\cdots,b_{r-1},0,\cdots,0)$. The above construction then implies that $\dim([A])\leq r-1$, a contradiction. So $\rank(\M_A)\geq r$. On the other hand, assume that $\L=\sum_{i=0}^r b_i \sigma^i$. Without loss of generality, we may assume that $b_r=1$. Then since $\L$ annihilates all entries of $\bfv_j$, one has that 
\begin{equation}
\label{EQ:order}
 \sigma^r(\bfv_j)=-\sum_{i=0}^{r-1}b_i\sigma^i(\bfv_j),\forall\,j=1,\cdots,n
\end{equation}
where $\bfv_j$ is the $j$-th column of $\F$. Applying $\sigma$ to (\ref{EQ:order}) successively yields that for each $l=0,\cdots,n^2-r$,
$$
   A_{r+l}\bfv_j=\sigma^{r+l}(\bfv_j)=\sum_{i=0}^{r-1}c_{l,i}\sigma^i(\bfv_j)=\sum_{i=0}^{r-1}c_{l,i}A_i \bfv_j, \,\forall\,j=1,\cdots,n
$$
where $c_{l,i}\in k(\bX)(x)$. Hence $A_{r+l}-\sum_{i=0}^{r-1}c_{l,i}A_i=0$, because $\F=(\bfv_1,\cdots,\bfv_n)$ is invertible. Consequently, the $(r+l)$-th row of $\M_A$ is a linear combination of the first $r$ rows of $\M_A$. Hence $\rank(\M_A)\leq r$. This proves that $\rank(\M_A)=r$. Similarly, one has that
$$\dim([A(\bfc)])=\rank(v_\bfc(\M_A))=\rank(\M_{A(\bfc)})$$ for all $\bfc\in \bX_\frakh$, where $\bX_\frakh$ is given as in Notation~\ref{NT:well-defined}.

Now take a nonzero $g\in k[\bX]$ such that $\rank(\M_A)=\rank(v_\bfc(\M_A))$ for any $\bfc\in \bX_g$. Then for $\bfc\in \bX_g\cap \bX_\frakh$, one has that
$$
  \dim([A])=\rank(\M_A)=\rank(v_\bfc(\M_A))=\rank(\M_{A(\bfc)})=\dim([A(\bfc)]).
$$
\end{proof}

Now let us turn to the dimension of $\rI(m,I_\nu)$. Let $\F=(f_{i,j})$ be a fundamental matrix of $\sigma_A(Y)=AY$ such that
$$
 I_\nu=\left\langle \left\{\left.p\in \overline{k(\bX)}(x)[X]_{\leq \nu} \,\,\right|\,\, p(\F)=0\right\} \right\rangle_{\overline{k(\bX)}(x)}.
$$
By Remark~\ref{RM:NT}, the vector space spanned by the entries of $\Sym_\nu(\calF)$ is equal to the one spanned by all $\prod f_{i,j}^{s_{i,j}}$ with $0\leq \sum s_{i,j}\leq \nu$.
Set
$$
   \calL_m^{\nu}(A)= \diag\left(\Sym_\nu(A),\left(\frac{x+1}{x}\right)\Sym_\nu(A), \cdots, \left(\frac{x+1}{x}\right)^m\Sym_\nu(A)\right)
$$
and
$$
   \tilde{\calF}=\diag\left(\Sym_\nu(\F),x\Sym_\nu(\F), \cdots, x^m\Sym_\nu(\F)\right).
$$
Note that
$$\sigma(\Sym_\nu(\calF))=\Sym_\nu(\sigma(\calF))=\Sym_\nu(A\calF)=\Sym_\nu(A)\Sym_\nu(\calF).$$
We have that $\tilde{\calF}$ is a fundamental matrix of $\sigma_A(Y)=\calL^\nu_m(A)Y$, and the set of the entries of $\tilde{\calF}$ and the set of all $x^i\prod f_{i,j}^{s_{i,j}}$ with $0\leq i \leq m$ and $0\leq \sum s_{i,j}\leq \nu$ span the same vector space. Notice that $$\rI(m,I_\nu)=\left\{ p\in \overline{k(\bX)}[x]_{\leq m}[X]_{\leq \nu} \,\,|\,\, p(\F)=0\right\}.$$
This implies that
\begin{equation}
\label{EQ:eqn1}
   \dim(\rI(m,I_\nu))=(m+1)\binom{n^2+\nu-1}{\nu}-\dim([\calL^\nu_m(A)]).
\end{equation}

\begin{corollary}
\label{COR:generators}
Let $m$ be a positive integer and $I_\nu$ be a $\nu$-maximal $\sigma_A$-ideal. Suppose that $B$ is a $\overline{k(\bX)}$-basis of $\rI(m,I_\nu)$ and $B\subset k[\bX][x,X]$. Then there is a nonempty open subset $U$ of $\bX$ such that for any $\bfc\in U$, $v_\bfc(B)$ is a basis of $\rI(m, \tilde{I}_\bfc)$ where $\tilde{I}_\bfc$ is a $\nu$-maximal $\sigma_{A(\bfc)}$-ideal in $k(x)[X,1/\det(X)]$.
\end{corollary}
\begin{proof}
By Proposition 1.20 on page 15 of \cite{vanderPut-Singer}, $I_\nu$ has a zero $\xi$ in $\GL_n(\overline{k(\bX)}(x))$. Write $B=\{b_1,\cdots,b_l\}$. Since $B$ generates $I_\nu$ that is a $\sigma_A$-ideal, there is a matrix $M$ with entries in $\overline{k(\bX)}(x)[X,1/\det(X)]$ such that
$$
    \sigma_A((b_1,\cdots,b_l))=(b_1,\cdots,b_l)M.
$$
Let $\tilde{D}\subset \overline{k(\bX)}$ be a finitely generated $k[\bX]$-algebra such that the entries of $\xi$ and the coefficients of the entries of $M$ are all in the fraction field of $\tilde{D}[x]$ and let $\bY$ be the variety over $k$ associated to $\tilde{D}$.
There is a nonzero $g\in \tilde{D}$ such that for any $\tilde{\bfc}\in \bY_{g}$, $v_{\tilde{\bfc}}(\xi), v_{\tilde{\bfc}}(M)$ are well-defined and $v_{\tilde{\bfc}}(\xi)$ is invertible. Then
\begin{align*}
   \sigma_{A(\tilde{\bfc})}((v_{\tilde{\bfc}}(b_1),\cdots,v_{\tilde{\bfc}}(b_l)))&=v_{\tilde{\bfc}}\left(\sigma_A((b_1,\cdots,b_l))\right)=v_{\tilde{\bfc}}\left((b_1,\cdots,b_l)M\right)\\
   &=(v_{\tilde{\bfc}}(b_1),\cdots,v_{\tilde{\bfc}}(b_l))v_{\tilde{\bfc}}(M).
\end{align*}
Hence for any $\tilde{\bfc}\in \bY_g$, $\langle v_{\tilde{\bfc}}(B)\rangle_{k(x)}$ is a $\sigma_{A(\tilde{\bfc})}$-ideal. Furthermore, $v_{\tilde{\bfc}}(\xi)$ is a zero of this ideal in $\GL_n(k(x))$. This implies that for such $\tilde{\bfc}$, $1\notin \langle v_{\tilde{\bfc}}(B)\rangle_{k(x)}$ and then $ v_{\tilde{\bfc}}(B)$ is contained in some $\nu$-maximal $\sigma_{A(\tilde{\bfc})}$-ideal, say $\tilde{I}_{\tilde{\bfc}}$, because every polynomial in $v_{\tilde{\bfc}}(B)$ is of degree in $X$ not greater than $\nu$. Using the arguments similar to those after Lemma~\ref{LM:dim}, one has that
  \begin{equation}
     \label{EQ:eqn2}
        \dim(\rI(m,\tilde{I}_{\tilde{\bfc}}))=(m+1)\binom{n^2+\nu-1}{\nu}-\dim([\calL_m^\nu(A(\tilde{\bfc}))]).
  \end{equation}
Let $\tilde{U}$ be a nonempty open subset of $\bY$ satisfying that for any $\tilde{\bfc}\in \tilde{U}$,
  \begin{itemize}
    \item [$(1)$]
        $\dim([\calL^\nu_m(A)])=\dim([v_{\tilde{\bfc}}(\calL^\nu_m(A))])$ and $v_{\tilde{\bfc}}(\calL^\nu_m(A))=\calL^\nu_m(A(\tilde{\bfc}))$; and
    \item [$(2)$] $v_{\tilde{\bfc}}(B)$ is linearly independent over $k$ and $|B|=|v_{\tilde{\bfc}}(B)|$.
  \end{itemize}
Such $\tilde{U}$ exists due to Lemma~\ref{LM:dim}.
Combining equalities (\ref{EQ:eqn1}) and (\ref{EQ:eqn2}), one sees that for any $\tilde{\bfc}\in \bY_g\cap \tilde{U}$,
$$
   | v_{\tilde{\bfc}}(B)|=|B|=\dim(\rI(m,I_\nu))=\dim(\rI(m,\tilde{I}_{\tilde{\bfc}})),
$$
which implies that $ v_{\tilde{\bfc}}(B)$ is a basis of $\rI(m,\tilde{I}_{\tilde{\bfc}}).$ The corollary then follows from Lemma~\ref{LM:kolchinlemma} and the fact that $v_{\tilde{\bfc}}(B)=v_{p_{\bY/\bX}(\tilde{\bfc})}(B)$. 
\end{proof}

\begin{proposition}
\label{PROP:maximalideals}
Let $I_\nu$ be a $\nu$-maximal $\sigma_A$-ideal and $N$ the integer obtained in Lemma~\ref{LM:coefficientbound}. Suppose that $\rI(N,I_\nu)$ has a $\overline{k(\bX)}$-basis $B$ contained in $k[\bX][x, X]$. Then there is a basic open subset $U$ of $\bX$ such that for any $\bfc\in U$, $v_\bfc(B)$ is a $k$-basis of $\rI(N,\tilde{I}_\bfc)$ for some $\nu$-maximal $\sigma_{A(\bfc)}$-ideal $\tilde{I}_\bfc$ in $k(x)[X,1/\det(X)]$. Specially, $v_\bfc(B)$ generates $\tilde{I}_\bfc$.
\end{proposition}
\begin{proof}
By Lemma~\ref{LM:coefficientbound}, there is a basic open subset $U_1$ of $\bX$ such that $N$ is a coefficient bound for not only $\nu$-maximal $\sigma_A$-ideals but also $\nu$-maximal $\sigma_{A(\bfc)}$-ideals for all $\bfc\in U_1$. By Corollary \ref{COR:generators}, there is a nonempty open subset $U_2$ of $\bX$ such that for any $\bfc\in  U_2$, one has that $v_\bfc(B)$ is a basis of $\rI(N, \tilde{I}_\bfc)$ for some $\nu$-maximal $\sigma_{A(\bfc)}$-ideal $\tilde{I}_\bfc$.
Set $U=U_1\cap U_2$. The proposition then follows from the fact that $\rI(N,\tilde{I}_\bfc)$ generates $\tilde{I}_\bfc$.
 \end{proof}

\section{Difference Galois groups under specialization}
\label{SEC:galoisgroups}
The aim of this section is to prove Theorem~\ref{TH:main}. To begin,  let us recall some notations and basic concepts in difference Galois theory.  Let $\frakm$ be a maximal $\sigma_A$-ideal of $\overline{k(\bX)}(x)[X,1/\det(X)]$ and let
$$
   \R=\overline{k(\bX)}(x)[X,1/\det(X)]/\frakm.
$$
Then $\R$ is the Picard-Vessiot ring of $\overline{k(\bX)}(x)$ for (\ref{EQ:differenceeqn}). The Galois group $\G$ of (\ref{EQ:differenceeqn}) over $\overline{k(\bX)}(x)$ is defined to be the set of $\overline{k(\bX)}(x)$-automorphisms of $\R$ which commute with $\sigma_A$. Set $\bar{X}=X \mod \frakm$. Then $\bar{X}$ is a fundamental matrix of (\ref{EQ:differenceeqn}), which induces a group homomorphism from $\G$ to $\GL_n(\overline{k(\bX)})$ given by sending $\phi\in \G$ to $\bar{X}^{-1}\phi(\bar{X})$. The image of this homomorphism is an algebraic subgroup of $\GL_n(\overline{k(\bX)})$ and this image can be obtained by computing the stabilizer of $\frakm$. The stabilizer of an ideal $I$ in $\overline{k(\bX)}(x)[X,1/\det(X)]$, denoted by $\stab(I)$, is defined to be the set of elements $g\in \GL_n(\overline{k(\bX)})$ satisfying that $\{p(Xg)|p\in I\}=I$, which is an algebraic subgroup of $\GL_n(\overline{k(\bX)})$. It is well-known that the stabilizer of $\frakm$ is the image of $\G$ under the homomorphism induced by a fundamental matrix that is a zero of $\frakm$. Throughout this section, Galois groups always mean the stabilizers of maximal $\sigma_A$-ideals. The readers are referred to Chapter 1 of \cite{vanderPut-Singer} for more details on difference Galois theory.
\subsection{A criterion for difference Galois groups}
\label{SUBSEC:criterion}
Proto-Galois groups play an essential role in the computation of difference Galois groups as well as differential Galois groups. In this subsection, we shall give a necessary and sufficient condition  for a proto-Galois group to be a difference Galois group. One will see that the condition given below can be verified algorithmically. Let us first recall what proto-Galois groups are.
\begin{definition}
\label{DEF:proto-groups}
Let $G, H$ be two algebraic subgroups of $\GL_n(\overline{k(\bX)})$. $H$ is said to be a proto-group of $G$ if it satisfies the following condition
   $$
       H^t \leq G^\circ \leq G \leq H
   $$
where  $H^t$ denotes the algebraic subgroup of $H$ generated by unipotent elements.
   In the case when $G$ is the Galois group of $\sigma_A(Y)=AY$ over $\overline{k(\bX)}(x)$, $H$ is called a proto-Galois group of $\sigma_A(Y)=AY$ over $\overline{k(\bX)}(x)$.
\end{definition}
\begin{remark}
\label{RMK:protogroup}
  \begin{itemize}
  \item [$(1)$] Since $H^t$ is connected, $H^t\subset H^\circ$. So if $H$ is a proto-group of $G$, then $H^\circ$ is a proto-group of $G\cap H^\circ$.
  \item [$(2)$] Suppose that $H$ is a proto-group of $G$ and $g\in \GL_n(\overline{k(\bX)})$. Then $H$ is a proto-group of $gGg^{-1}$ if and only if $gGg^{-1}\subset H$. To see this, it suffices to prove the ``if" part. Note that if $h\in \GL_n(\overline{k(\bX)})$ is unipotent then so is $g h g^{-1}$. Thus $gH^t g^{-1}\subset H^t$, because $gH^tg^{-1}\subset gGg^{-1}\subset H$. As both $gH^tg^{-1}$ and $H^t$ are connected and have the same dimension, $gH^tg^{-1}=H^t$. This implies that
      $
         H^t=gH^t g^{-1}\subset gGg^{-1} \subset H.
      $
  \item [$(3)$] Suppose that $H$ is a proto-Galois group of $\sigma_A(Y)=AY$ over $\overline{k(\bX)}(x)$ and $A\in H(\overline{k(\bX)}(x))$. Let $\tilde{H}$ be an algebraic subgroup of $H$. We claim that if $\sigma(h^{-1})Ah\in \tilde{H}(\overline{k(\bX)}(x))$ for some $h\in \GL_n(\overline{k(\bX)}(x))$ then $H$ is a proto-group of $\tilde{H}$.  Let $G$ be the Galois group of $\sigma_A(Y)=AY$ over $\overline{k(\bX)}(x)$ satisfying that $H$ is a proto-group of $G$. Proposition 1.21 of \cite{vanderPut-Singer} implies that there is $g\in \GL_n(\overline{k(\bX)})$ such that $gGg^{-1}\subset \tilde{H}$. By $(2)$, $H$ is a proto-group of $gGg^{-1}$ and then it is a proto-group of $\tilde{H}$ by the definition. This proves the claim.
  \end{itemize}
\end{remark}

Let $H$ be an algebraic subgroup of $\GL_n(\overline{k(\bX)})$ such that $A\in H(\overline{k(\bX)}(x))$. It was proved in Proposition 1.21 of \cite{vanderPut-Singer} that $H$ is the Galois group of (\ref{EQ:differenceeqn}) over $\overline{k(\bX)}(x)$ if and only if for any $g\in H(\overline{k(\bX)}(x))$ and any proper algebraic subgroup $\tilde{H}$ of $H$ one has that $\sigma(g^{-1})A g\notin \tilde{H}(\overline{k(\bX)}(x))$. We shall refine this criterion when $H$ is a proto-Galois group of (\ref{EQ:differenceeqn}) over $\overline{k(\bX)}(x)$. As an analogue of finite algebraic extensions in differential case, we need to consider the power of $\sigma$. Let $i$ be a positive integer. Obviously, every $\sigma$-ring (resp. field) is also a $\sigma^i$-ring (resp. field) and an easy calculation yields that $\sigma_A^i(X)=A_iX$, where $A_i$ stands for $\sigma^{i-1}(A)\cdots \sigma(A)A$.
\begin{definition}
Let $s\geq 0$. The rational functions $a_1,\cdots,a_m\in \overline{k(\bX)}(x)\setminus \{0\}$ are said to be multiplicatively $\sigma^s$-independent if for any $d_i\in \bZ$ and any $f\in \overline{k(\bX)}(x)\setminus \{0\}$, $\prod_{i=1}^m a_i^{d_i}=\sigma^s(f)/f$ implies that $d_1=\cdots=d_m=0$.
\end{definition}
\begin{lemma}
\label{LM:connectedgroup}
Let $H$ be a connected algebraic subgroup of $\GL_n(\overline{k(\bX)})$ and $B \in H(\overline{k(\bX)}(x))$. Suppose that $H$ is a proto-Galois group of $\sigma_B(Y)=B Y$ over $\overline{k(\bX)}(x)$. Then $H$ is the Galois group of $\sigma_B(Y)=BY$ over $\overline{k(\bX)}(x)$ if and only if
$\{\chi(B)|\chi\in \frakX\}$ is multiplicatively $\sigma$-independent, where $\frakX$ is a basis of $\bfchi(H)$.
\end{lemma}
\begin{proof}
Suppose that $H$ is the Galois group and there are integers $d_\chi, \chi\in \frakX$, not all zero, such that
$$\prod_{\chi\in \frakX} \chi^{d_\chi}(B)=\frac{\sigma(f)}{f}$$
for some $f\in \overline{k(\bX)}(x)\setminus \{0\}$. Set $\chi=\prod_{\chi\in\frakX} \chi^{d_\chi}$. Then $\chi$ is a nontrivial character. Let $I$ be the ideal in $\overline{k(\bX)}(x)[X,1/\det(X)]$ generated by all vanishing polynomials of $H$. Since $B \in H(\overline{k(\bX)}(x))$ and $H$ is the Galois group, $I$ is a maximal $\sigma_B$-ideal (see Lemma 1.10 and its proof on page 8 of \cite{vanderPut-Singer}). Furthermore as $H$ is connected, $I$ is a prime ideal. Let $\bar{X}= X \mod I$ and $E=\overline{k(\bX)}(x)(\bar{X})$. Then $\bar{X}$ is a fundamental matrix of $\sigma_B(Y)=B Y$ and it belongs to $H(E)$. An easy calculation yields that $\sigma_B(\chi(\bar{X}))/\chi(\bar{X})=\sigma(f)/f$ and then
  $\sigma_B\left(\chi(\bar{X})f^{-1}\right)=\chi(\bar{X})f^{-1}.$
In other words, $\chi(\bar{X})f^{-1}$ is a constant of $E$.
Since $E$ is the total Picard-Vessiot ring of $\sigma_B(Y)=BY$ and $\overline{k(\bX)}$ is algebraically closed, the field of constants of $E$ is equal to $\overline{k(\bX)}$. Hence
$\chi(\bar{X})=cf$ for some $c\in \overline{k(\bX)}$.
This implies that $\chi(X)-cf\in I$. As elements of $H(\overline{k(\bX)}(x))$ are zeroes of $I$, putting $X=I_n$ in $\chi(X)-cf$ yields that $cf=1$, and then putting $X=B $ in $\chi(X)-1$ yields that $\chi(B)=1$, i.e. $B\in \ker(\chi)$. Proposition 1.21 on page 15 of \cite{vanderPut-Singer} implies that $\ker(\chi)$ contains $H$ as a subgroup. Hence $\ker(\chi)=H$, i.e. $\chi$ is trivial. This contradicts the fact that $\chi$ is nontrivial.

Conversely, suppose that $H$ is not the Galois group. Due to Proposition 1.21 of \cite{vanderPut-Singer} again, there is $g\in H(\overline{k(\bX)}(x))$ and a proper algebraic subgroup $\tilde{H}$ of $H$ such that $\sigma(g^{-1})B g\in \tilde{H}(\overline{k(\bX)}(x))$. By Remark~\ref{RMK:protogroup}, $H$ is a proto-group of $\tilde{H}$. By Proposition 2.6 of \cite{feng2}, $\tilde{H}\subset \ker(\chi)$ for some nontrivial character $\chi$ of $H$. This implies that $\chi(\sigma (g^{-1})B g)=1$, i.e. $\chi(B)=\sigma(\chi(g))/\chi(g)$. Consequently, $\chi(B), \chi\in \frakX$ are multiplicatively $\sigma$-dependent.
\end{proof}
Remark that the above lemma still holds if we replace $\sigma, B$ and $\sigma_B$ with $\sigma^s, B_s$ and $\sigma_B^s$ respectively for some positive integer $s$.
Now let us consider the general case.
\begin{proposition}
\label{PROP:criterion}
 Let $H$ be an algebraic subgroup of $\GL_n(\overline{k(\bX)})$ such that $A\in H(\overline{k(\bX)}(x))$. Suppose that $H$ is a proto-Galois group of $\sigma_A(Y)=AY$ over $\overline{k(\bX)}(x)$. Then $H$ is the Galois group of $\sigma_A(Y)=AY$ over $\overline{k(\bX)}(x)$ if and only if
\begin{itemize}
  \item [$(a)$] $A_m \notin H^\circ(\overline{k(\bX)}(x))$ for all positive $m$ with $m |\ell$ and $m\neq \ell$, and
  \item  [$(b)$]
$\{
   \chi\left(A_\ell\right) | \chi\in\frakX
\}$
is multiplicatively $\sigma^\ell$-independent,
\end{itemize}
where $\ell=[H:H^\circ]$ and $\frakX$ is a basis of $\chi(H^\circ)$.
\end{proposition}
\begin{proof}
Assume that $H$ is the Galois group of $\sigma_A(Y)=AY$ over $\overline{k(\bX)}(x)$. 
Let $I$ be a maximal $\sigma_A$-ideal in $\overline{k(\bX)}(x)[X,1/\det(X)]$ such that $H=\stab(I)$, the stabilizer of $I$. For each positive integer $m$, note that $I$ is a proper $\sigma_A^m$-ideal in $\overline{k(\bX)}(x)[X,1/\det(X)]$, so there is a maximal $\sigma_A^m$-ideal, say $\tilde{I}_m$, containing $I$. By Lemma 4.1 of \cite{feng2}, $\tilde{I}_m\cap \sigma_A(\tilde{I}_m)\cap \dots\cap\sigma_A^{m-1}(\tilde{I}_m)$ is a maximal $\sigma_A$-ideal. It is clear that each $\sigma_A^i(\tilde{I}_m)$  contains $I$ as so does $\tilde{I}_m$. Thus 
$$I=\tilde{I}_m\cap \sigma_A(\tilde{I}_m)\cap \dots\cap\sigma_A^{m-1}(\tilde{I}_m)$$
because $I$ is a maximal $\sigma_A$-ideal. Denote $H_m=\stab(\tilde{I}_m)$. Then $H_m$ is the Galois group of $\sigma_A^m(Y)=A_mY$ over $\overline{k(\bX)}(x)$. Due to Lemma 4.1 of \cite{feng2} agian, $H_m$ is a subgroup of finite index in $H$ and furthermore $[H:H_m]\leq m$. This implies that $H_m$ contains $H^\circ$ by Proposition on page 53 of \cite{humphreys}.	
Now if $A_m\in H^\circ(\overline{k(\bX)}(x))$ for some positive $m$ with $m|\ell$ and $m\neq \ell$, then by Proposition 1.21 of \cite{vanderPut-Singer}, $H_m$ is a subgroup of $H^\circ$. This implies that $H_m=H^\circ$ and thus $[H:H_m]=\ell$, a contradiction with $[H:H_m]\leq m<\ell$. Therefore $A_m\notin H^\circ(\overline{k(\bX)}(x))$ for all positive $m$ with $m|\ell$ and $m\neq \ell$, i.e. $(a)$ holds. In addition, note that $\sigma^i(A)\in H(\overline{k(\bX)}(x))$ for all $i\geq 0$. From this, one sees that $A_\ell=\sigma^{\ell-1}(A)\cdots A\in H^\circ(\overline{k(\bX)}(x))$. By Proposition 1.21 of \cite{vanderPut-Singer} again, $H_\ell$ is a subgroup of $H^\circ$. However, $H_\ell$ contains $H^\circ$. This implies that  $H_\ell=H^\circ$ and $H^\circ$ is the Galois group of $\sigma_A^\ell(Y)=A_\ell Y$ over $\overline{k(\bX)}(x)$. Then Lemma~\ref{LM:connectedgroup} with $\sigma_B=\sigma_A^\ell$ implies $(b)$. This proves the necessary part.

It remains to show that $(a)$ and $(b)$ are sufficient. Suppose to the contrary that $H$ is not the Galois group under the assumption that $(a)$ and $(b)$ hold. By Proposition 1.21 on page 15 of \cite{vanderPut-Singer}, there are $g\in H(\overline{k(\bX)}(x))$ and a proper algebraic subgroup $\tilde{H}$ of $H$ such that $\sigma(g^{-1})Ag\in \tilde{H}(\overline{k(\bX)}(x))$. Write $g=h\xi$ with $h\in H^\circ(\overline{k(\bX)}(x))$ and $\xi\in H(\overline{k(\bX)})$. Then for $i>0$
\begin{equation}
\label{EQ:powers}
   \sigma^i(g^{-1})A_i g=\prod_{j=0}^{i-1}\sigma^{i-1-j}\left(\sigma(g^{-1})Ag\right)=\xi^{-1}\sigma^i(h^{-1})A_i h \xi.
\end{equation}
We claim that the condition $(b)$ implies that $\tilde{H}^\circ=H^\circ$. To see this, suppose that $H^\circ\neq \tilde{H}^\circ$. Setting $i=\ell$ in (\ref{EQ:powers}), one has that
$$\sigma^\ell(g^{-1})A_\ell g\in \tilde{H}(\overline{k(\bX)}(x))\cap H^\circ(\overline{k(\bX)}(x)).$$
Notice that $H$ is a proto-group of $\tilde{H}$ as shown in Remark~\ref{RMK:protogroup}. Thus $H^\circ$ is a proto-group of $\tilde{H}\cap H^\circ$. Furthermore, since $\tilde{H}^\circ\neq H^\circ$, $\tilde{H}\cap H^\circ$ is a proper subgroup of $H^\circ$. Due to Proposition 2.6 of \cite{feng2}, there is a nontrivial character $\chi\in \bfchi(H^\circ)$ such that $\tilde{H}\cap H^\circ\subset \ker(\chi)$, and so
$
   \chi\left( \xi^{-1}\sigma^\ell(h^{-1})A_\ell h \xi\right)=1.
$
Set $\tilde{\chi}=\chi(\xi^{-1}X\xi)$. Then $\tilde{\chi}$ is still a nontrivial character of $H^\circ$ and
$\tilde{\chi}(A_\ell)=\sigma^\ell(\tilde{\chi}(h))/\tilde{\chi}(h)$. Write $\tilde{\chi}=\prod_{\chi\in\frakX} \chi^{d_\chi}$ where $d_\chi\in \bZ$ and not all of them are zero. Then one sees that
$\{\chi(A_\ell) | \chi\in \frakX\}$ is not multiplicatively $\sigma^\ell$-independent,
which contradicts the condition $(b)$. Hence $H^\circ=\tilde{H}^\circ$. This proves the claim. Now let $m=[\tilde{H}:\tilde{H}^\circ]$. Then $m|\ell$ and setting $i=m$ in (\ref{EQ:powers}) yields that
\begin{equation*}
   \sigma^m(g^{-1})A_m g=\xi^{-1}\sigma^m(h^{-1})A_m h \xi\in \tilde{H}^\circ(\overline{k(\bX)}(x))=H^\circ(\overline{k(\bX)}(x)).
\end{equation*}
So $A_m\in \sigma^m(h)\xi H^\circ(\overline{k(\bX)}(x))\xi^{-1} h^{-1}$. As $\xi H^\circ \xi^{-1}=H^\circ$  and $h\in H^\circ(\overline{k(\bX)}(x))$,
$
A_m\in  H^\circ(\overline{k(\bX)}(x)).
$
The assumption $(a)$ then implies that $m=\ell$, i.e. $\tilde{H}=H$. This contradicts the assumption that $\tilde{H}$ is a proper subgroup of $H$. Therefore $H$ is the Galois group.
\end{proof}
\begin{remark}
\label{RM:criterion}
Let $H$ and $A$ be as in Proposition~\ref{PROP:criterion}.
\begin{itemize}
\item [$(1)$]
If $H$ is connected, i.e. $\ell=[H:H^\circ]=1$, then the condition $(a)$ always holds and the proposition reduces to Lemma~\ref{LM:connectedgroup}.
\item [$(2)$] Assume that $\tilde{H}$ is the Galois group of $\sigma_A(Y)=AY$ over $\overline{k(\bX)}(x)$. From the proof of the sufficient part of the proposition, one sees that $(b)$ implies $H^\circ=\tilde{H}^\circ$. We claim that the converse is also true. Suppose that $H^\circ=\tilde{H}^\circ$. Since $\tilde{H}$ is a subgroup of $H$ by Proposition 1.21 of \cite{vanderPut-Singer}, $[H:\tilde{H}]|\ell$. Denote $m_1=[H:\tilde{H}]$ and $m_2=[\tilde{H}:\tilde{H}^\circ]=\ell/m_1$. By Lemma 1.26 and Corollary 1.17 of \cite{vanderPut-Singer}, $\tilde{H}^\circ$ is the Galois group of $\sigma_A^{m_2}(Y)=A_{m_2}Y$ over $\overline{k(\bX)}(x)$. Note that $\sigma_A^\ell=(\sigma_A^{m_2})^{m_1}$ and $A_\ell=(A_{m_2})_{m_1}$. Applying Lemma 4.1 of \cite{feng2} to $\sigma_A^{m_2}(Y)=A_{m_2}Y$ yields that $[\tilde{H}^\circ:\tilde{H}_\ell]\leq m_1$, where $\tilde{H}_\ell$ is the Galois group of $\sigma_A^\ell(Y)=A_\ell Y$ over $\overline{k(\bX)}(x)$. Hence $\tilde{H}_\ell=\tilde{H}^\circ=H^\circ$. Lemma~\ref{LM:connectedgroup} with $\sigma_B=\sigma_A^\ell$ then implies $(b)$. This proves our claim.
\end{itemize}
\end{remark}

\subsection{Proof of Theorem~\ref{TH:main}}
\label{SUBSEC:proofmaintheorem}
Before we prove the following proposition, let us first recall some results in \cite{feng2}. Note that the reference \cite{feng2} used some different notations, for instance $\nu$-maximal $\sigma_A$-ideals are denoted by $I_{\F,\nu}$ and the stabilizer of $I_{\F,\nu}$ is denoted by $H_{\F,\nu}$. By Proposition 3.10 of \cite{feng2}, the stabilizer of a $\nu$-maximal $\sigma_A$-ideal with sufficiently large $\nu$ is a proto-Galois group of $\sigma_A(Y)=AY$ over $\overline{k(\bX)}(x)$. Precisely, let $\nu$ be an integer greater than the integer $\tilde{d}$ given in Proposition 2.5 of \cite{feng2}, and $I_\nu$ a $\nu$-maximal $\sigma_A$-ideal. Suppose that $I$ is a maximal $\sigma_A$-ideal containing $I_\nu$. Let $G=\stab(I)$ and $H=\stab(I_\nu)$ where $\stab()$ denotes the stabilizer. Then $G$ is the Galois group of $\sigma_A(Y)=AY$ over $\overline{k(\bX)}(x)$. Proposition 3.7 of \cite{feng2} implies that $\Zero(I_\nu)$ and $\Zero(I)$ are trivial $\overline{k(\bX)}(x)$-torsors for $H(\overline{k(\bX)(x)})$ and $G(\overline{k(\bX)(x)})$ respectively, where $\Zero()$ denotes the set of zeroes in $\GL_n(\overline{k(\bX)(x)})$. Let $g\in \Zero(I)\cap \GL_n(\overline{k(\bX)}(x))$. Then
$$
   \Zero(I_\nu)=gH(\overline{k(\bX)(x)})\supset \Zero(I)=gG(\overline{k(\bX)(x)}).
$$
Thus $G\subset H$ and moreover $H$ is a proto-group of $G$.
\begin{proposition}
\label{PROP:protoGalois}
Let $G$ be the Galois group of $\sigma_A(Y)=AY$ over $\overline{k(\bX)}(x)$. Assume that $A\in G(k(\bX)(x))$ and $G$ is defined over $k[\bX]$.
Then there is a basic open subset $U$ of $\bX$ such that $G_\bfc$ is a proto-Galois group of $\sigma_{A(\bfc)}(Y)=A(\bfc)Y$ over $k(x)$ for any $\bfc\in U$, where $G_\bfc$ is defined as in Section~\ref{SEC:algebraicgroups}.
\end{proposition}
\begin{proof}
Let $d$ be an integer greater than the integer $\tilde{d}$ given in Proposition 2.5 of \cite{feng2}. Let $S\subset k[\bX][X]$ be a finite set generating the vanishing ideal of $G$, and let $I$ be the ideal in $\overline{k(\bX)}(x)[X,1/\det(X)]$ generated by $S$. Since $A\in G(k(\bX)(x))$ and $G$ is the Galois group, $I$ is a maximal $\sigma_A$-ideal (see Lemma 1.10 and its proof on page 8 of \cite{vanderPut-Singer}). Suppose that $m$ is a positive integer such that all polynomials in $S$ are of total degree in $X$ not greater than $m$.  Set
$$
      \nu=\max\left\{m, d\right\}.
$$
Then $I$ is a $\nu$-maximal $\sigma_A$-ideal. Due to Lemma~\ref{LM:coefficientbound}, there is a coefficient bound of $I$, say $N$, and a basic open subset $U_1$ of $\bX$ such that for every $\bfc\in U_1$, $N$ is also a coefficient bound of $\nu$-maximal $\sigma_{A(\bfc)}$-ideals in $k(x)[X,1/\det(X)]$.
Let $B$ be a basis of $\rI(N,I)$ where $\rI(N,I)$ is defined as in (\ref{EQ:partial}). Let $\tilde{D}\subset \overline{k(\bX)}$ be a finitely generated $k[\bX]$-algebra such that $B\subset \tilde{D}[x,X]$ and let $\bY$ be the variety over $k$ associated to $\tilde{D}$. Because $S$ and $B$ generate the same ideal $I$, using an argument similar to that in Remark~\ref{RM:corollaries}, one can prove that there is a nonempty open subset $\tilde{U}_1$ of $\bY$ such that for each $\tilde{\bfc}\in \tilde{U}_1$, $v_{\tilde{\bfc}}(S)$ and $v_{\tilde{\bfc}}(B)$ generates the same ideal in $k(x)[X,1/\det(X)]$.
By Proposition~\ref{PROP:maximalideals}, there is a basic open subset $\tilde{U}_2$ of $\bY$ such that for any $\tilde{\bfc}\in \tilde{U}_2$, $v_{\tilde{\bfc}}(B)$  is a $k$-basis of $\rI(N,\tilde{I}_{\tilde{\bfc}})$ for some $\nu$-maximal $\sigma_{A(\tilde{\bfc})}$-ideal $\tilde{I}_{\tilde{\bfc}}$ in $k(x)[X,1/\det(X)]$. Specially, $v_{\tilde{\bfc}}(B)$ generates $\tilde{I}_{\tilde{\bfc}}$. Then for any $\tilde{\bfc}\in \tilde{U}_1\cap \tilde{U}_2$, $v_{\tilde{\bfc}}(B)$ and $v_{\tilde{\bfc}}(S)$ generate the same ideal $\tilde{I}_{\tilde{\bfc}}$. By Lemma~\ref{LM:property}, there is a basic open subset $U_2$ of $\bX$ that is contained in $p_{\bY/\bX}(\tilde{U}_1\cap \tilde{U}_2)$. Due to Proposition~\ref{PROP:groupsfibre}, there is a basic open subset $U_3$ of $\bX$ such that for any $\bfc\in U_3$, $v_\bfc(S)$ defines an algebraic subgroup $G_\bfc$ of $\GL_n(k)$. Now set $U=U_1\cap U_2\cap U_3$ and suppose $\bfc\in U$. Let $\tilde{\bfc}\in \tilde{U}_1\cap \tilde{U}_2\cap p_{\bY/\bX}^{-1}(\bfc)$. Then $G_\bfc(\overline{k(x)})$ is the variety in $\GL_n(\overline{k(x)})$ defined by $\tilde{I}_{\tilde{\bfc}}$ that is generated by $v_{\bfc}(S)(=v_{\tilde{\bfc}}(S))$.
Let $H=\stab(\tilde{I}_{\tilde{\bfc}})$. Since $\tilde{I}_{\tilde{\bfc}}$ is $\nu$-maximal, due to Proposition 3.7 of \cite{feng2}, $G_\bfc(\overline{k(x)})$ is a trivial $k(x)$-torsor for $H(\overline{k(x)})$. As $I_n\in G_\bfc$ and both $G_\bfc$ and $H$ are defined over $k$, we have that $G_\bfc=H$, i.e. $G_\bfc$ is the stabilizer of $\tilde{I}_{\tilde{\bfc}}$. Proposition 3.10 of \cite{feng2} and the choice of $\nu$ then imply that $G_\bfc$ is a proto-Galois group of $\sigma_{A(\tilde{\bfc})}(Y)=A(\tilde{\bfc})Y$ over $k(x)$. The proposition then follows from the fact that $A(\tilde{\bfc})=A(\bfc)$.
\end{proof}

Suppose that $\bfa=(a_1,\cdots,a_m)$ with $a_i\in \overline{k(\bX)}(x)\setminus\{0\}$ and $\ell\geq 0$. Denote 
$$
   \calZ(\bfa,\ell)=\left\{\bfd=(d_1,\cdots,d_m)\in \bZ^m \left| \exists\,f\in \overline{k(\bX)}(x)\setminus\{0\} \,\,\mbox{s.t.}\,\, \bfa^\bfd=\frac{\sigma^\ell(f)}{f}\right.\right\},
$$
where $\bfa^\bfd=a_1^{d_1}\cdots a_m^{d_m}$.
Then $\calZ(\bfa,\ell)$ is a finitely generated $\bZ$-module. We say $a_i$ is $\ell$-standard if for any $\alpha, \beta$ in the set of zeroes and poles of $a_i$, $\alpha-\beta\in \ell \bZ$ implies that $\alpha=\beta$. One has that if $a_i\notin \overline{k(\bX)}$ then $\sigma^\ell(a_i)/a_i$ is not $\ell$-standard. To see this, write $a_i=\lambda\prod_{j=1}^s (x-c_j)^{d_j}$ where $\lambda, c_1,\cdots,c_s\in \overline{k(\bX)}$, $\lambda\neq 0$, $c_{j_1}\neq c_{j_2}$ if $j_1\neq j_2$ and all $d_j$ are nonzero integers. Then
$$\frac{\sigma^\ell(a_i)}{a_i}=\prod_{j=1}^s \frac{(x-(c_j-\ell))^{d_j}}{(x-c_j)^{d_j}}.$$
Set $m_1=\min\{l| \exists \,\,c_i\,\,\mbox{s.t.}\,\,c_1=c_i-l\ell\}$ and $m_2=\max \{l| \exists \,\,c_i\,\,\mbox{s.t.}\,\,c_1=c_i-l\ell\}$.
Then both $x-(c_1+(m_1-1)\ell)$ and $x-(c_1+m_2\ell)$ can not be cancelled in $\sigma^\ell(a_i)/a_i$. That is to say, both $c_1+(m_1-1)\ell$ and $c_1+m_2\ell$ are in the set of zeroes and poles of $\sigma^\ell(a_i)/a_i$. As the difference of $c_1+(m_1-1)\ell$ and $c_1+m_2\ell$ is equal to $(m_1-m_2-1)\ell$ that is a nonzero element in $\ell\bZ$, $\sigma^\ell(a_i)/a_i$ is not $\ell$-standard.
\begin{lemma}
\label{LM:multiindependent}
Suppose that $\bfa=(a_1,\cdots,a_m)$ with $a_i\in k(\bX)(x)\setminus\{0\}$ and $\ell\geq 0$. Then there is a basic open subset $U$ of $\bX$ such that for any $\bfc\in U$, $a_1(\bfc), \cdots,a_m(\bfc)$ are well-defined and
$
   \calZ(\bfa,\ell)=\calZ(v_{\bfc}(\bfa),\ell).
$
\end{lemma}
\begin{proof}
Let $W$ be the set of zeroes and poles of $a_1,\cdots,a_m$ in $\overline{k(\bX)}$, and let $\bfalpha\subset W$ be the representative of $W$ in the quotient group $\overline{k(\bX)}/\ell\bZ$. Suppose $\beta\in W$. Then $\beta=\alpha+\ell d$ for some $\alpha\in \bfalpha$ and $d\in \bZ$. If $d=0$, set $g=1$, otherwise set
\[
    g=\begin{cases}
        \prod_{l=1}^d (x-\alpha-\ell l)^{-1} & d>0 \\
        \prod_{l=0}^{-d-1} (x-\alpha+\ell l) & d<0
    \end{cases}.
\]
Then $x-\beta=\sigma^\ell(g)(x-\alpha)/g$. Under the multiplication with $\sigma^\ell(g)/g$, we can replace $x-\beta$ by $x-\alpha$ for all $a_i$. Hence
for every $i=1,\cdots,m$, we can write
$$
a_i=\xi_i\frac{\sigma^\ell(f_i)}{f_i} \prod_{\alpha\in \bfalpha} (x-\alpha)^{e_{i,\alpha}}
$$
where $\xi_i\in k(\bX)\setminus\{0\}$, $e_{i,\alpha}\in \bZ$ and $f_i\in \overline{k(\bX)}(x)\setminus\{0\}$ whose numerator and denominator are both monic. Set $\bar{a}_i=\prod_{\alpha\in\bfalpha} (x-\alpha)^{e_{i,\alpha}}$ for all $i=1,\cdots,m$. One sees easily that $\bfa^\bfd=\sigma^\ell(f)/f$ if and only if $\bfxi^\bfd=1$ and $\bar{\bfa}^\bfd=\sigma^\ell(\tilde{f})/\tilde{f}$, where $\bfxi=(\xi_1,\cdots,\xi_m)$ and $\bar{\bfa}=(\bar{a}_1,\cdots,\bar{a}_m)$. Since $\bar{\bfa}^\bfd$ is $\ell$-standard, if $\bar{\bfa}^\bfd=\sigma^\ell(\tilde{f})/\tilde{f}$ then $\tilde{f}\in \overline{k(\bX)}$. Therefore $\bfa^\bfd=\sigma^\ell(f)/f$ if and only if $\bfxi^\bfd=1$ and $\bar{\bfa}^\bfd=1$. Namely,
$$
  \calZ(\bfa,\ell)=\calZ(\bfxi,0)\cap \calZ(\bar{\bfa},0).
$$
Let $\Gamma_1$ be the subgroup of $\bG_m(k(\bX))$ generated by $\xi_1,\cdots,\xi_m$.
Let $\tilde{D}\subset \overline{k(\bX)}$ be a finitely generated $k[\bX]$-algebra such that $\Gamma_1, W\subset \tilde{D}$, and let $\bY$ be the variety over $k$ associated to $\tilde{D}$. Let $\Gamma_2$ be the subgroup of $\bG_a(\tilde{D})$ generated by $\{1\}\cup \bfalpha$. Now assume that $\tilde{\bfc}\in \B(\bY,\Gamma_1)\cap \B(\bY,\Gamma_2)$. Then
$$
    a_i(\tilde{\bfc})=\xi_i(\tilde{\bfc})\frac{\sigma^\ell(f_i(\tilde{\bfc}))}{f_i(\tilde{\bfc})} \prod_{\alpha\in \bfalpha} (x-\alpha(\tilde{\bfc}))^{e_{i,\alpha}},
$$
and moreover $\alpha(\tilde{\bfc})-\alpha'(\tilde{\bfc})\notin \ell \bZ$ if $\alpha\neq \alpha'$. A similar argument as above implies that $v_{\tilde{\bfc}}(\bfa)^\bfd=\sigma^\ell(f')/f'$ if and only if $v_{\tilde{\bfc}}(\bfxi)^\bfd=1$ and $v_{\tilde{\bfc}}(\bar{\bfa})^\bfd=1$. In other words,
$$
    \calZ(v_{\tilde{\bfc}}(\bfa),\ell)=\calZ(v_{\tilde{\bfc}}(\bfxi),0)\cap \calZ(v_{\tilde{\bfc}}(\bar{\bfa}),0).
$$
Since $\tilde{\bfc}\in \B(\bY,\Gamma_1)$,
$
  \calZ(v_{\tilde{\bfc}}(\bfxi),0)=\calZ(\bfxi,0).
$
Moreover, one has that $\calZ(v_{\tilde{\bfc}}(\bar{\bfa}),0)=\calZ(\bar{\bfa},0)$ for both of them are equal to
$$\left\{(d_1,\cdots,d_m)\in \bZ^m \left| \sum_{i=1}^m d_i e_{i,\alpha}=0, \,\,\forall\,\,\alpha\in \bfalpha\right.\right\}.$$
Consequently,
$
\calZ(\bfa,\ell)=\calZ(v_{\tilde{\bfc}}(\bfa),\ell).
$
 Lemma~\ref{LM:property} then completes the proof.
\end{proof}
\begin{corollary}
\label{COR:multiindependent}
Let $\bfa=(a_1,\cdots,a_m), \ell$ be as in Lemma~\ref{LM:multiindependent}. Then there is a basic open subset $U$ of $\bX$ such that for any $\bfc \in U$,  $a_1,\cdots,a_m$ are multiplicatively $\sigma^\ell$-independent if and only if so are $a_1(\bfc),\cdots,a_m(\bfc)$.
\end{corollary}
\begin{proof}
Note that $a_1,\cdots,a_m$ are multiplicatively $\sigma^\ell$-independent if and only if $\calZ(\bfa,\ell)=\{(0,\cdots,0)\}$. Corollary then follows from Lemma~\ref{LM:multiindependent}.
\end{proof}
Now we are ready to prove Theorem~\ref{TH:main}.
\begin{proof}[Proof of Theorem~\ref{TH:main}]
By Theorem 2.7 of \cite{hendriks}, there is $g\in \GL_n(\overline{k(\bX)}(x))$ such that $\sigma(g^{-1})Ag\in G(\overline{k(\bX)}(x))$. Denote $\tilde{A}=\sigma(g^{-1})Ag$. It is well-known that $\sigma_{\tilde{A}}(Y)=\tilde{A}Y$ and $\sigma_A(Y)=AY$ have the same Galois group. Let $D'\subset \overline{k(\bX)}$ be a finitely generated $k[\bX]$-algebra with $F'$ as field of fractions such that $g\in \GL_n(F'(x))$ and let $\bX'$ be the variety over $k$ associated to $D'$. Then there is $c'\in D'$ such that for any $\bfc'\in \bX'_{c'}$, both $g(\bfc')$ and $A(\bfc')$ are well-defined and invertible. For such $\bfc'$,
 $\sigma_{A(\bfc')}(Y)=A(\bfc')Y$ and $\sigma_{\tilde{A}(\bfc')}(Y)=\tilde{A}(\bfc')Y$ have the same Galois group. Remark that $\tilde{A}\in G(k(\bX')(x))$. Suppose that the theorem holds for $\sigma_{\tilde{A}}(Y)=\tilde{A}Y$ and $V'$ is the corresponding basic open subset of $\bX'$. Then for $\bfc'\in V'\cap \bX'_{c'}$, $G_{\bfc'}$ is the Galois group of $\sigma_{A(\bfc')}(Y)=A(\bfc')Y$ over $k(x)$. By Lemma~\ref{LM:property}, there is a basic open subset $V$ of $\bX$ contained in $p_{\bX'/\bX}(V'\cap \bX'_{c'})$. From the fact that $A(\bfc')=A(p_{\bX'/\bX}(\bfc'))$ and $G_{\bfc'}=G_{p_{\bX'/\bX}(\bfc')}$, one has that $G_\bfc$ is the Galois group of $\sigma_{A(\bfc)}(Y)=A(\bfc)Y$ over $k(x)$ for all $\bfc\in V$. Consequently, one only need to prove the theorem for  the case with $A\in G(k(\bX)(x))$.

Let $\frakX\subset \overline{k(\bX)}[X,1/\det(X)]$ be a basis of $\bfchi(G^\circ)$, and let $T$ be a finite set in $k(\bX)[X,1/\det(X)]$ generating the vanishing ideal of $G^\circ$. Let $\tilde{D}\subset \overline{k(\bX)}$ be a finitely generated $k[\bX]$-algebra such that $T, \frakX \subset \tilde{D}[X,1/\det(X)]$ and $\bY$ the variety over $k$ associated to $\tilde{D}$. Set $\ell=[G:G^\circ]$. Since $G$ is the Galois group of $\sigma_A(Y)=AY$ over $\overline{k(\bX)}(x)$ and $A\in G(k(\bX)(x))$, Proposition~\ref{PROP:criterion} implies that $A_m\notin G^\circ(\overline{k(\bX)}(x))$ for all positive $m$ with $m|\ell$ and $m\neq \ell$, and $\{\chi(A_\ell) | \chi\in \frakX\}$ is multiplicatively $\sigma^\ell$-independent. Thus, for all such $m$, there is $q_m\in T$ such that $q_m(A_m)\neq 0$. By Propositions~\ref{PROP:protoGalois}, ~\ref{PROP:groupsfibre} and ~\ref{PROP:character}, there is a basic open subset $\tilde{U}_1$ of $\bY$ such that for any $\tilde{\bfc}\in \tilde{U}_1$, one has that
\begin{itemize}
\item [$(a)$]
    $G_{\tilde{\bfc}}$ is a proto-Galois group of $\sigma_{A(\tilde{\bfc})}(Y)=A(\tilde{\bfc})Y$ over $k(x)$, and
\item [$(b)$]
   $[G_{\tilde{\bfc}}:G_{\tilde{\bfc}}^\circ]=[G:G^\circ]=\ell$, and
\item [$(c)$]
    $v_{\tilde{\bfc}}(\frakX)$ is a basis of $\bfchi(G_{\tilde{\bfc}}^\circ)$.
\end{itemize}
By Corollary~\ref{COR:multiindependent}, there is a basic open subset $\tilde{U}_2$ of $\bY$ such that for any $\tilde{\bfc}\in \tilde{U}_2$,
$\{v_{\tilde{\bfc}}(\chi(A_\ell))| \chi \in \frakX\}$ is multiplicatively $\sigma^\ell$-independent. Let $c$ be a nonzero element in $\tilde{D}$ such that for any $\tilde{\bfc}\in \bY_c$,
$v_{\tilde{\bfc}}(q_m(A_m))\neq 0$ for all positive $m$ with $m|\ell$ and $m\neq \ell$. Set
$$
\tilde{U}=\tilde{U}_1\cap \tilde{U}_2  \cap \bY_c
$$
and assume that $\tilde{\bfc}\in U$.
Since $v_{\tilde{\bfc}}(\chi(A_\ell))=v_{\tilde{\bfc}}(\chi)(A(\tilde{\bfc})_\ell)$, $\{v_{\tilde{\bfc}}(\chi)(A(\tilde{\bfc})_\ell)| \chi\in \frakX\}$ is multiplicatively $\sigma^\ell$-independent, that is to say, $\{\bar{\chi}(A(\tilde{\bfc})_\ell)| \bar{\chi}\in v_{\tilde{\bfc}}(\frakX)\}$ is multiplicatively $\sigma^\ell$-independent.
On the other hand, for all positive $m$ with $m|\ell$ and $m\neq \ell$, since $v_{\tilde{\bfc}}(q_m)(A(\tilde{\bfc})_m)=v_{\tilde{\bfc}}(q_m(A_m))\neq 0$, $A(\tilde{\bfc})_m \notin G_{\tilde{\bfc}}^\circ(k(x))$. By Proposition~\ref{PROP:criterion}, $G_{\tilde{\bfc}}$ is the Galois group of $\sigma_{A(\tilde{\bfc})}(Y)=A(\tilde{\bfc})Y$ over $k(x)$. The theorem then follows from Lemma~\ref{LM:property} and the fact that $A(\tilde{\bfc})=A(p_{\bY/\bX}(\tilde{\bfc}))$ and $G_{\tilde{\bfc}}=G_{p_{\bY/\bX}(\tilde{\bfc})}$.
\end{proof}

\begin{example}
\label{example3}Consider the linear difference equation $\sigma_A(Y)=AY$ with
$$
   A=\begin{pmatrix}
            x & t_1 x & 0 \\
            x & x & 0\\
            0  & 0 & t_2
   \end{pmatrix}
$$
where $t_1, t_2$ are parameters. Set $k=\bQ$ and $\bX=k^2$. Then $k(\bX)=k(t_1,t_2)$ and $A\in \GL_n(k(\bX)(x))$. Let
$
   S=\{X_{11}-X_{22}, X_{12}-t_1X_{21},X_{13},X_{23},X_{31},X_{32}\},
$
and denote by $H$ the variety in $\GL_3(\overline{k(\bX)})$ defined by $S$, i.e.
$$
  H=\left\{\left. \begin{pmatrix}
          a & t_1b &  0  \\
          b & a  & 0\\
          0 & 0 & c
  \end{pmatrix}\right | a,b,c \in \overline{k(\bX)}, c(a^2-t_1b^2)\neq 0 \right\}.
$$
One can verify that $H$ is connected and a basis of $\bfchi(H)$ can be represented by
$$
 \frakX= \{ \chi_1=X_{11}-\sqrt{t_1}X_{21}, \chi_2=X_{11}+\sqrt{t_1}X_{21}, \chi_3=X_{33}\}.
$$
Furthermore, one can verify that $A\in H(\overline{k(\bX)}(x))$ and $H$ is the Galois group of $\sigma_A(Y)=AY$ over $\overline{k(\bX)}(x)$. We shall find a basic open subset $U$ of $\bX$ such that $H_\bfc$ is the Galois group of $\sigma_{A(\bfc)}=A(\bfc)Y$ over $k(x)$ for all $\bfc\in U$. For the sake of simplicity, at some steps, we do not follow the proofs of preceding lemmas or propositions to get the corresponding basic open sets.

First of all, $A(\bfc)$ is invertible only if $\bfc\in \bX_{(t_1-1)t_2}$. Moreover, if $\bfc\in \bX_{(t_1-1)t_2}$, $H_\bfc$ is a connected algebraic subgroup of $\GL_3(k)$. It is easy to see that if $\bfc\in \bX_{t_1(t_1-1)t_2}$ then $A(\bfc)\in H_\bfc(k(x))$ and $H_\bfc^t=\{\id\}$. Thus for such $\bfc$, $H_\bfc$ is a proto-Galois group of $\sigma_{A(\bfc)}(Y)=A(\bfc)Y$ over $k(x)$.

Second, since $\chi_i$ is defined over $k(\sqrt{t_1},t_2)$, we need to extend $k[\bX]$ to $k[t_1,t_2,\sqrt{t_1}]$ whose associated variety we denote by $\bY$. Because $H^t=\{\id\}$ and $H_{\tilde{\bfc}}^t=\{\id\}$ for any $\tilde{\bfc}\in \bY_{t_1(t_1-1)t_2}$, the proof of Proposition~\ref{PROP:character} implies that if $v_{\tilde{\bfc}}(\frakX)$ is multiplicatively independent then it is a basis of $\bfchi(H_{\tilde{\bfc}})$. In the proof of Proposition~\ref{PROP:character}, take $g=(g_{i,j})$ with
$g_{1,1}=g_{2,2}=5/2, g_{1,2}=t_1g_{2,1}=\sqrt{t_1}/2, g_{3,3}=5$ and other entries being zero. Then one has that $\chi_1(g)=2, \chi_2(g)=3, \chi_3(g)=5$. From this, one sees that if $\tilde{\bfc}\in \bY_{t_1(t_1-1)t_2}$ then $v_{\tilde{\bfc}}(\frakX)$ is multiplicatively independent and thus it is a basis of $\bfchi(H_{\tilde{\bfc}})$.

Third, denote $\bfa=(\chi_1(A),\chi_2(A),\chi_3(A))$ where $$\chi_1(A)=x(1-\sqrt{t_1}), \chi_2(A)=x(1+\sqrt{t_1}),\chi_3(A)=t_2.$$
Take $\bfxi=(1-\sqrt{t_1}, 1+\sqrt{t_1}, t_2)$ and $\bar{\bfa}=(x,x,1)$. It is easy to see that $\calZ(\bfxi,0)=\{(0,0,0)\}$. Let $\tilde{\Gamma}$ be the subgroup of $\bG_m(k(\bY))$ generated by $1-\sqrt{t_1}, 1+\sqrt{t_1}, t_2$. Then for any $\tilde{\bfc}\in \B(\bY,\tilde{\Gamma})$, $\calZ(v_{\tilde{\bfc}}(\bfxi),0)=\{(0,0,0)\}$ and therefore
$$
   \calZ(v_{\tilde{\bfc}}(\bfa),1)=\calZ(v_{\tilde{\bfc}}(\bfxi),0)\cap \calZ(v_{\tilde{\bfc}}(\bar{\bfa}),0)=\{(0,0,0)\}
$$
 i.e. $\{v_{\tilde{\bfc}}(\chi_i(A))|i=1,2,3\}$ is multiplicatively $\sigma$-independent. Let $\tilde{U}=\bY_{t_1(t_1-1)t_2}\cap \B(\bY,\tilde{\Gamma})$. Then for any $\tilde{\bfc}\in \tilde{U}$, since $H_{\tilde{\bfc}}$ is connected, Lemma~\ref{LM:connectedgroup} implies that $H_{\tilde{\bfc}}$ is the Galois group of $\sigma_{A(\tilde{\bfc})}(Y)=A(\tilde{\bfc})Y$ over $k(x)$.

 Finally, by (\ref{eqn:basicopensets}), $\B(\bX,\tilde{\Gamma})=p_{\bY/\bX}(\B(\bY,\tilde{\Gamma}))$ and so $p_{\bY/\bX}(\tilde{U})$ contains $\B(\bX,\tilde{\Gamma})\cap \bX_{t_1(t_1-1)t_2}$. The latter set is what we need.
\end{example}
\begin{example}
\label{example4}
Consider
$$\sigma_t(y)=ty$$
over $\CX(x,t)$ where $t$ is a parameter. This equation has a solution $t^x$.
\begin{enumerate}
\item
 $t$ is endowed with the usual derivation $\partial_t$. The differential Galois group is $\bG_m(\CX)$ and Proposition 2.9 of \cite{hardouin-singer} implies that $t^x$ satisfies a first order linear differential equation over $\CX(x,t)$ with respect to $\partial_t$. Actually, one easily sees that $t^x$ is a solution of $\partial_t(y)=(x/t)y$.
\item $t$ is endowed with the shift operator $\tau(t)=t+1$. Example 3.8 of \cite{ovchinnikov-wibmer} implies that $t^x$ does not satisfy any nonzero difference equation over $\CX(x,t)$ with respect to $\tau$.
\item $t$ is a usual parameter. The usual Galois group is $\bG_m(\overline{\CX(t)})$ and it implies that $t^x$ is transcendental over $\CX(x,t)$. One sees that $c^x$ is algebraic over $\CX(x)$ if and only if $c$ is a root of unity. In particular, when $c=1$, the Galois group of the specialized equation is $\{\id\}$.
\end{enumerate}
\end{example}
The above example provides one a glance at the difference between parameterized difference Galois theories and difference Galois theory with usual parameters. Remark that for higher order linear difference equations, the phenomenon appearing in (3) of Example~\ref{example4} can also happen i.e. the Galois group of the specialized equation is extremely small under some specialization even though the original one is as large as the whole general linear group.
\section{An application}
\label{SEC:application}
In this section, we apply Theorem~\ref{TH:main} to the inverse problem in difference Galois theory. The notations are as before, for instance $k$ denotes an algebraically closed field of characteristic zero, $\sigma_B$ with $B\in \GL_n(k(x))$ denotes the $k$-automorphism of $k(x)[X,1/\det(X)]$ induced by $\sigma_B(X)=BX$ and $\sigma(x)=x+1$,  $v_\bfc$ denotes the map from $k[\bX]$ to $k$ given by $v_\bfc(f)=f(\bfc)$ for $f\in k[\bX]$ and $\stab(I)$ stands for the stabilizer of an ideal $I$. The inverse problem asks which algebraic subgroups of $\GL_n(k)$ occur as the Galois groups of $\sigma_B(Y)=BY$ over $k(x)$ with $B\in \GL_n(k(x))$. In Chapter 3 of \cite{vanderPut-Singer}, van der Put and Singer raised the following conjecture.
\begin{conjecture}
\label{CON1}
An algebraic subgroup $G$ of $\GL_n(k)$ is the Galois group of $\sigma_B(Y)=BY$ over $k(x)$ for some $B\in \GL_n(k(x))$ if and only if $G/G^\circ$ is cyclic.
\end{conjecture}
It was shown in Proposition 1.20 of \cite{vanderPut-Singer} that $G/G^\circ$ is necessary to be cyclic if $G$ is the Galois group of $\sigma_B(Y)=BY$ over $k(x)$. Therefore, to prove Conjecture~\ref{CON1}, it suffices to prove the sufficient part, which we restate as the following conjecture.
\begin{conjecture}
\label{CON2}
If $G$ is an algebraic subgroup of $\GL_n(k)$ satisfying that $G/G^\circ$ is cyclic then $G$ is the Galois group of $\sigma_B(Y)=BY$ over $k(x)$ for some $B\in \GL_n(k(x))$.
\end{conjecture}
When $k=\CX$, for connected algebraic groups and cyclic extensions of tori, analytic proofs of Conjecture~\ref{CON2} were presented in Corollary 8.6 and Lemma 8.12 of \cite{vanderPut-Singer}, respectively. In Chapter 3 of the same book, an algebraic proof of Conjecture~\ref{CON2} was also given when $k$ is any algebraically closed field of characteristic zero and $G$ is connected. For the general case, Conjecture~\ref{CON2} remains open.


Using a similar argument as that in the proof of Theorem 4.4 of \cite{singer}, we  can prove the following theorem.
\begin{theorem}
\label{TH:inverseproblem}
If Conjecture~\ref{CON2} holds for $k=\CX$, then it holds for any algebraically closed field $k$ of characteristic zero.
\end{theorem}
\begin{proof}
Let $G$ be an algebraic subgroup of $\GL_n(k)$ with $G/G^\circ$ cyclic. Suppose that the vanishing ideal of $G$ is generated by a finite set $S\subset k[X,1/\det(X)]$. Assume that the cardinality of $k$ is at most the cardinality of $\CX$. Then we can assume that $k\subset \CX$, and thus $G(\CX)$ is an algebraic subgroup of $\GL_n(\CX)$ with  $G(\CX)/G^\circ(\CX)$ cyclic. The assumption implies that $G(\CX)$ is the Galois group of $\sigma_B(Y)=BY$ over $\CX(x)$ for some $B\in \GL_n(\CX(x))$. Without loss of generality, we may assume that $B\in G(\CX(x))$. Let $D\subset \CX$ be a finitely generated $k$-algebra such that the entries of $B$ are all in the field of fractions of $D[x]$, and let $\bX$ be the variety over $k$ associated to $D$. We claim that $G(\overline{k(\bX)})$ is the Galois group of $\sigma_B(Y)=BY$ over $\overline{k(\bX)}(x)$. Otherwise, by Proposition 1.21 of \cite{vanderPut-Singer}, there is $T\in G(\overline{k(\bX)}(x))$ and a proper $\overline{k(\bX)}$-subgroup $H$ of $G(\overline{k(\bX)})$ such that $\sigma(T)BT^{-1}\in H(\overline{k(\bX)}(x))$. Since $\overline{k(\bX)}\subset \CX$, $H(\CX)$ is a proper subgroup of $G(\CX)$ and $T\in G(\CX(x))$. By Proposition 1.21 of \cite{vanderPut-Singer} again, $G(\CX)$ is not the Galois group of $\sigma_B(Y)=BY$ over $\CX(x)$, a contradiction. This proves our claim. Due to Theorem~\ref{TH:main}, there is $\bfc\in \bX$ such that $G_\bfc$, the variety in $\GL_n(k)$ defined by $v_\bfc(S)$, is the Galois group of $\sigma_{B(\bfc)}(Y)=B(\bfc)Y$ over $k(x)$. On the other hand, since $S\subset k[X,1/\det(X)]$, $S=v_\bfc(S)$ and then $G=G_\bfc$. Thus $G$ is the Galois group of $\sigma_{B(\bfc)}(Y)=B(\bfc)Y$ over $k(x)$.

Now assume that the cardinality of $k$ is larger than the cardinality of $\CX$. Then we can assume that $\CX\subset k$ and $G$ is defined over $\CX$.
By the assumption again, $G(\CX)$ is the Galois group of $\sigma_B(Y)=BY$ over $\CX(x)$ for some $B\in \GL_n(\CX(x))$. Let $I$ be a maximal $\sigma_B$-ideal of $\CX(x)[X,1/\det(X)]$ such that $G(\CX)=\stab(I)$ and let $\tilde{I}$ be the ideal in $k(x)[X,1/\det(X)]$ generated by $I$. Due to Proposition 2.4 of \cite{chatzidakis-hardouin-singer}, $\tilde{I}$ is a maximal $\sigma_B$-ideal. One can verify that $\stab(\tilde{I})=G$. So $G$ is the Galois group of $\sigma_B(Y)=BY$ over $k(x)$.
\end{proof}
The above theorem together with Corollary 8.6 and Lemma 8.12 of \cite{vanderPut-Singer} implies the following corollary.
\begin{corollary}
Conjecture~\ref{CON2} holds when $G$ is a connected affine algebraic group or a cyclic extension of a torus.
\end{corollary}




\end{document}